\documentclass[a4paper]{amsart}

\usepackage{amssymb}
\usepackage{amsmath}

\title{Q-prime curvature on CR manifolds}
\author{Kengo Hirachi}
\thanks{
Graduate School of Mathematical Sciences, 
The University of Tokyo,
3-8-1 Komaba, Megro, Tokyo 153-8914 JAPAN}



\numberwithin{equation}{section}

\newtheorem{lem}{Lemma}[section]
\newtheorem{thm}[lem]{Theorem}
\newtheorem{prop}[lem]{Proposition}

\newtheorem{defn}[lem]{\it Definition}
\newtheorem{rem}[lem]{\it Remark}

\newcommand{\calE}{\mathcal{E}}

\newcommand{\calH}{\mathcal{H}}

\newcommand{\calJ}{\mathcal{J}}
\newcommand{\calK}{\mathcal{K}}
\newcommand{\calL}{\mathcal{L}}

\newcommand{\calN}{\mathcal{N}}
\newcommand{\calO}{\mathcal{O}}
\newcommand{\calP}{\mathcal{P}}

\newcommand{\calPE}{\mathcal{PE}}

\newcommand{\bfP}{\mathbf{P}}
\newcommand{\bfS}{\mathbf{S}}

\newcommand{\bC}{\mathbb{C}}

\newcommand{\bP}{\mathbb{P}}
\newcommand{\bR}{\mathbb{R}}

\newcommand{\bZ}{\mathbb{Z}}

\newcommand{\bth}{\boldsymbol{\th}}
\newcommand{\bh}{\boldsymbol{h}}

\newcommand{\SU}{\operatorname{SU}}

\renewcommand{\Re}{\operatorname{Re}}
\renewcommand{\Im}{\operatorname{Im}}
\newcommand{\Ric}{\operatorname{Ric}}
\newcommand{\contr}{\operatorname{contr}}

\newcommand{\tf}{\operatorname{tf}}
\newcommand{\lp}{\operatorname{lp}}
\newcommand{\Scal}{\operatorname{Scal}}

\newcommand{\wt}{\widetilde}
\newcommand{\wf}{\wt f}

\newcommand{\rs}{r_\sharp}

\newcommand{\pa}{\partial}
\newcommand{\wh}{\widehat}
\renewcommand{\th}{\theta}
\newcommand{\up}{\Upsilon}
\newcommand{\conj}{\overline}
\newcommand{\ab}{{\alpha\conj\beta}}
\newcommand{\CR}{\lrcorner\,}

\begin{document}
\maketitle
\begin{center}
{\em Dedicated to Mike Eastwood on the occasion of his 60th birthday}
\end{center}

\begin{abstract}
$Q$-prime curvature, which was introduced by J. Case and P. Yang, is a local invariant of pseudo-hermitian structure on CR manifolds that can be defined only when the $Q$-curvature vanishes identically.  It is considered as a secondary invariant on CR manifolds and, in 3-dimensions, its integral agrees with the Burns-Epstein invariant, a Chern-Simons type invariant in CR geometry.  We give an ambient metric construction of the $Q$-prime curvature and study its basic properties.  In particular, we show that, for the boundary of a strictly pseudoconvex domain in a Stein manifold, the integral of the $Q$-prime curvature is a global CR invariant, which generalizes the Burns-Epstein invariant to higher dimensions.
\end{abstract}

\section{Introduction}
Tom Branson introduced the concept of $Q$-curvature in conformal geometry 
around 1990 \cite{Br} in connection with the study of conformal anomaly of determinants of conformally invariant differential operators. Since then $Q$-curvature has played an increasing and central role in geometric analysis on conformal manifolds. 
Around the same time, we have introduced in \cite{H1} a pseudo-hermitian invariant on CR manifolds, which was coincidentally denoted by $Q$, in a study of the asymptotic expansion of the Szeg\"o kernel of strictly pseudoconvex domains. 
The CR version of $Q$ was later shown to agree with Branson's $Q$-curvature defined with respect to Fefferman's conformal structure on a circle bundle over CR manifolds
 \cite{FH}.  Using this correspondence, one can translate the properties of conformal $Q$-curvature to the CR analogue;
  see \cite{GG}, \cite{HPT}.
However, there has been an important missing piece in this correspondence.  
In conformal geometry, the integral of the $Q$-curvature, called the total $Q$-curvature, is a global conformal invariant
and its first variation under the deformation of conformal structure is give by the Fefferman-Graham obstruction tensor \cite{GH}. On the other hand, the total CR $Q$-curvature always vanishes for domains in $\bC^N$
and has no relation to the obstruction function, which arises in the asymptotic analysis of the complex Monge-Amp\`ere equation \cite{F1}, \cite{Gr2}; see also \eqref{asymp-exp}. Moreover, CR $Q$-curvature identically vanishes for a natural choice of contact forms, called pseudo-Einstein contact forms, on the boundary of a domain in $\bC^N$; see \cite{FH}, \cite{CC}. 

We claim that the missing piece can be filled by $Q$-prime curvature, which was first introduced by Case and Yang \cite{CaY}  on $3$-dimensional CR manifolds in the process of generalizing the work of Branson, Fontana and Morpurgo \cite{BFM} on the conditional intertwinors ($P$-prime operators defined below) acting on CR pluriharmonic functions on the CR sphere $S^{2n+1}$; see also Remark \ref{CYrem}.
The idea is very simple: the vanishing of $Q$-curvature for pseudo-Einstein contact form $\th$ enables us to define a secondary local invariant of $\th$. 
 We give here a heuristic argument for defining $Q$-prime curvature.
We start by recalling the definition of $Q$-curvature in CR geometry.  Let $D\subset\bC^{n+1}$ be a strictly pseudoconvex domain with $C^\infty$ boundary $M=\pa D$.  Then, by Cheng and Yau \cite{ChY}, $D$ admits a defining function $r$ (positive in $D$ and $dr\ne0$ on $\pa D$)  such that $-i\pa\conj\pa\log r$ is Einstein-K\"ahler with negative scalar curvature, which we normalized to be $-n-2$.  The defining function $r$ also gives a metric on the trivial bundle $\bC^*\times\conj D$:  Denoting the fiber variable by $z^0\in\bC^*$, we define a Ricci-flat Lorentz K\"ahler metric, called 
{\em the ambient metric}, on $\bC^*\times\conj D$
by
$$
\wt g=-{i}\pa\conj\pa\big( |z^0|^2 r(z)\big).
$$
We can use this metric to construct CR invariant differential operators on $M$. Let $\wt\Delta$ be the Laplacian of $\wt g$. Then, for an integer $2m\in[-n,0]$ and a function $f$ on $\conj D$, 
$$
\Big(\wt\Delta^{n+2m+1} |z^0|^{2m} f\Big)\Big|_{\bC^*\times M},
$$ 
is shown to depend only on the boundary value of $f$ and is homogeneous in $z^0$.
Since the functions on $\bC^*\times M$ with homogeneity in $z^0$ can be identified with  densities on $M$, we obtain a family of CR invariant differential operators
$$
P_{n+2m+1}\colon\calE(m)\to\calE(-n-m-1).
$$
Note that, if $m=0$, then $1\in\calE(0)$ and $P_{n+1}1=0$. Branson's idea is to consider the $0$th order term of
$P_{n+1}$ for higher dimensions $N+1$ and take the ``limit as $N\to n$'' after factoring out $(N-n)$. A formal definition of $Q$-curvature is
$$
Q_n=\lim_{N\to n}\frac{1}{N-n}\Big(\wt\Delta^{n+1} |z^0|^{2(n-N)}\Big)\Big|_{\bC^*\times M^{2N+1}}.
$$
One way to justify this limit is to consider Taylor expansion in $N-n$.  Using
$$
|z^0|^{2(n-N)}=e^{(n-N)\log|z^0|^2}=\sum_{k=0}^\infty\frac{(N-n)^k}{k!}
(-\log |z^0|^2)^k,
$$
we have a ``formal expansion''
$$
\big(\wt\Delta^{n+1}|z^0|^{2(n-N)}\big)|_{\bC^*\times M^{2n+1}}=\sum_{k=0}^\infty\frac{(N-n)^k}{k}
Q^{(k)},
$$ 
where
$$
Q^{(k)}=\wt\Delta^{n+1}(-\log |z^0|^2)^k\big|_{\bC^*\times M}.
$$
While the expansion does not have clear meaning, the coefficients $Q^{(k)}$ are standard quantities defined on an $(n+2)$-dimensional Lorentz K\"ahler manifold.
Clearly, $Q^{(0)}=0$. The second term $Q^{(1)}$ gives the CR $Q$-curvature, which can be also considered as the $Q$-curvature of the Lorentz metric on
$S^1\times M$ given by the restriction of $\wt g$.
However, $Q^{(1)}=0$ because $\log |z^0|^2$ is pluriharmonic.
Hence the leading term of the expansion is $Q^{(2)}$,
which we define to be the {\em Q-prime curvature} and denote by $Q'$.
We can see that $Q'$ is a pseudo-hermitian invariant of the contact form
$\th=(i/2)(\pa-\conj\pa)r|_M$, which is
pseudo-Einstein (Definition \ref{pe-def}) due to the Ricci-flatness of $\wt g$.
This definition of $Q'$ can be generalized to embedded CR manifolds with pseud-Einstein contact forms (Definition \ref{def-Qp}).

 A similar argument can be used to define a new differential operator acting on the boundary values of pluriharmonic functions, or CR pluriharmonic functions;
we denote the space of such functions by $\calP$. If $f\in\calP$, then its pluriharmonic extension $\wf$ is uniquely determined; moreover, at each point of $M$, the Taylor series of $\wt f$ is determined by that of $f$ by solving $\pa\conj\pa \wt f=0$. So we can define a nontrivial differential operator on $\calP$ by
$$
P'f=-\wt\Delta^{n+1}(\wf\log |z^0|^2)\big|_{\bC^*\times M},
$$
which we call the {\em $P$-prime operator} (Definition \ref{def-pp}).

The definitions of $Q'$ and $P'$ depend on the choice of $r$ and are not CR invariant.  Since the K\"ahler potential has the ambiguity of adding pluriharmonic functions, the choice of $r$ also has the ambiguity
$\wh r=e^\up r$, where $\up$ is pluriharmonic on $\conj D$.  Under this change of defining function, we have transformation rules (Propositions \ref{trans-PP-prop} and  \ref{transQP})
\begin{equation*}
\begin{aligned}
\wh P'f&=P'f+P_{n+1}(\up f),
\\
\wh Q'\ &=Q'+2P'\up+P_{n+1}(\up^2).
\end{aligned}
\end{equation*}
A crucial fact is that $P'$ and $P_{n+1}$ are formally self-adjoint respectively on
$\calP$ and $C^\infty(M)$ (Theorem \ref{selfP}), and $P'1=P_{n+1}1=0$.  It follows that
the integral, which we call the {\em total $Q$-prime curvature},
$$
\conj{Q}'(M)=\int_M Q'
$$
is a CR invariant of $M$. 
By analogy with the fact that the total $Q$-curvature in conformal geometry is given by the logarithmic term in the asymptotic expansion of the volume of conformally compact Einstein manifold,
one may hope that $\conj{Q}'$ is a coefficient of the asymptotic volume expansion of  $D$.  It is indeed the case, but is not with respect to the volume form $dv_g$ of $g$.  We will show, in Theorem \ref{totalQpthm}, that $\conj{Q}'$ appears in the  expansion with respect to the volume form weighted by $\|d\log r\|^2$, the squared norm of  the 1-form $d\log r$ for $g$:
\begin{equation}\label{asym-tPQ}
\int_{r>\epsilon}
\|d\log r\|^2dv_g=\sum_{j=0}^{n}a_j\,\epsilon^{j-n-1}+
c_n\,\conj{Q}'\,\log\epsilon+O(1),
\end{equation}
where $c_n=(-1)^n/(n!)^{3}$.
This formula can be applied to compute the variation of $\conj{Q}'$ under the deformation of CR structures,
which generalizes the formula \eqref{QPvar}  in dimension $3$ stated below.  The detail will  appear in our forthcoming paper with Yoshihiko Matsumoto and Taiji Marugame.

In the case $M$ has dimension 3, we can explicitly write down $P'$ and $Q'$ in terms of the Tanaka-Webster connection. With respect to the contact form $\th=({i}/{2})(\pa-\conj\pa)r|_M$ for $r$ given as above, we have 
\begin{align}\label{PP3}
P'f&=\Delta_b^2f-\Re\nabla^1(\Scal \nabla_1f-2iA_{11}\nabla^1f),
\\
\label{Q4form}
Q'&=\Delta_b \Scal+\frac12 \Scal^2-2|A|^2.
\end{align}
Here $\Delta_b$ is the sub-Laplacian, $\nabla$ is the Tanaka-Webster connection,
and $\Scal$, $A$ are respectively the scalar curvature and torsion of the connection.
(These formulas were first derived by Case and Yang.)
It turns out that the total $Q$-prime curvature agrees with the Burns-Epstein invariant $\mu(M)$, \cite{BE1}, up to a universal constant multiple:
$$
\conj{Q}'(M)=\int_M(\frac12 \Scal^2-2|A|^2)\th\wedge d\th=-8\pi^2\mu(M).
$$
From this fact we can also obtain the renormalized Gauss-Bennet formula:
\begin{equation}\label{GBthm}
\int_D c_2(R_g)-\frac13 c_1(R_g)=\chi(D)-\frac{1}{2\pi^2}\conj{Q}'(M),
\end{equation}
where $c_k(R_g)$ is the $k$th Chern form defined from  $g$; see \cite{BE2}.
Another consequence is the variational formula for a smooth family of strictly pseudoconvex domains $\{D_{t}\}_{t\in\bR}$ in $\bC^2$.  If $D=D_0$, then $D_t$ for small $t$ can be parametrized by a density
$f\in\calE(1)$ on $M=\pa D$ and the first variation of $\conj{Q}'(M_t)$,
$M_t=\pa D_t$, is given by
\begin{equation}\label{QPvar}
\left.\frac{d}{dt}\right|_{t=0}\conj{Q}'(M_t)=2\Re \int_M f\calO,
\end{equation}
where $\calO$ is the obstruction function of Fefferman; see \eqref{asymp-exp} and
\S 6.5.
These two properties are good evidence that $Q'$ is a natural object in CR geometry and contains important information that cannot be captured by the $Q$-curvature, which comes from the conformal geometry.

We should mention the generalization of the Burns-Epstein invariant to higher dimensions by themselves \cite{BE2}.  The invariants are defined by renormalizing the Chern class of complete Einstein-K\"ahler metric on the strictly pseudoconvex domains in $\bC^{n+1}$ as in \eqref{GBthm}. The construction uses a transgression formula given on the Cartan bundle over the boundary; hence it is not easy to compare the integrand with $Q'$.  Recently, a much more explicit formula for the transgression is found by
Marugame \cite{M}; but the relation to $Q'$ is yet to be studied.

In this introduction, for simplicity, we have formulated the results for the boundaries of strictly pseudoconvex domains in $\bC^{n+1}$. But we prove all results for domains in Stein manifolds, or in some cases under weaker assumptions; see \S7 for  such generalizations.  We also need to mention the fact that the defining function $r$ given by Cheng-Yau possibly has weak singularity at the boundary.
In fact, $r$ has asymptotic expansion at the boundary
\begin{equation}\label{asymp-exp}
r\sim \rho+\rho\sum_{j=1}^\infty \eta_j(\rho^{n+2}\log\rho)^j,
\end{equation}
where $\rho$ is a smooth defining function and $\eta_j\in C^\infty(\conj D)$; $\eta_1|_{\pa D}$ is called the obstruction function, which, in case $n=1$, agrees with the local CR invariant $\calO$ given in \eqref{QPvar} up to a non-zero constant multiple.
In the ambient metric construction of $Q'$, the numbers of the derivatives applied to $r$ is limited and the logarithmic terms do not contribute; so $r$ can be replaced by the smooth part $\rho$.
Thus, in this paper, instead of the Cheng-Yau solution $r$, we use a smooth defining function that satisfies the Einstein equation approximately at the boundary; see Proposition \ref{Ricci-prop}.
The smooth approximate solution is given by an explicit algebraic algorithm of Fefferman
\cite{F1} and unique modulo  $O(\rho^{n+3})$.  It implies in particular  that $Q'$ is locally determined by $\th$.
More detailed study of the expansion \eqref{asymp-exp}  is given in \cite{LM}, \cite{Gr2} and \cite{H2}.

This paper is organized as follows.  
In \S2, we give a review of the ambient metric for CR manifolds.
Then, in \S3, we construct CR invariant operators via the ambient space,
which include GJMS (or Gover-Graham) operators between CR densities and 2-from valued invariant operator on functions that characterizes CR pluriharmonic functions.  We define $P$-prime operator in \S4 and prove its self-adjointness. In
\S5, we define $Q$-prime curvature and study its basic properties: transformation law and invariance of its integral.  In \S6, we study the 3-dimensional case: explicitly write down $P'$ and $Q'$; describe the relation between $Q'$ and the Burns-Epstein integrand.  In the final section, \S7, we give an observation on Hartogs' extension theorem for pluriharmonic functions, which are used in the proof of the self-adjointness of $P'$. 

\begin{rem}\label{CYrem}\rm
The aim of Case and Yang \cite{CaY} is  to study $P'$ on 3-dimensional CR manifolds as an analogy of Paneitz operator  on conformal  4-dimensional manifolds as in Chang, Gursky and Yang \cite{ChGY}. They derive $P'$ and $Q'$ by following Branson's formal procedure: firstly, compute $P_2$ in all dimensions  by using tractor calculus of \cite{GG} and then formally setting the dimension to $3$.  All the results in \cite{CaY} are based on the explicit formulas of $P'$ and $Q'$ as given above; so the justification of Branson's argument is favorable but is not  inevitable in their work.
\end{rem}

\medskip

\noindent
{\em Notations.}
We use Einstein's summation convention and assume that 
\begin{itemize}
\item
uppercase Latin indices $I,J,K,\dots$ run from 0 to $n+1$;
\item
lowercase Latin indices
$j,k,l,\dots$
run from 1 to $n+1$; 
\item
lowercase Greek indices $\alpha,\beta,\gamma,\dots$ run from 1 to $n$. 
\end{itemize}
The letter $i$ denotes $\sqrt{-1}$.

\section{The ambient metric in CR geometry}

We start with a review of CR geometry and the ambient metric associated with it.
We basically follow Graham-Gover \cite{GG}, but we here try to make the definition of ambient metric more intrinsic to the complex structure. 
In the construction of $P'$ and $Q'$, the K\"ahler condition on the ambient metric is essential, and we mostly confine ourselves to the embedded CR manifolds.  

\subsection{The ambient metric}
Let $M$ be a $C^\infty$ manifold of dimension $2n+1$, $n\ge1$, and $\bC TM$ be the complexified tangent bundle.  A {\em CR structure} on $M$ is a complex $n$-dimensional subbundle $T^{1,0}$ of $\bC TM$ such that $T^{1,0}\cap T^{0,1}=\{0\}$, where $T^{0,1}=\conj{T^{1,0}}$.  In the following, we assume that $T^{1,0}$ is {\em integrable} in the sense that the sections of $T^{1,0}$ are closed under Lie bracket.  We set $\calH=\Re T^{1,0}$ and assume that there is a real one form $\th$ such that $\ker\th=\calH$.
The {\em Levi form} of $\th$ is the hermitian form on  $T^{1,0}$ defined by
$$
L_\th(Z,W)=-id\th(Z,\conj{W}).
$$
Under the scaling $\wh\th=e^\up\th$, $\up\in C^\infty(M)$,  we have
$$
L_{\wh\th}=e^{\up}L_\th.
$$
Thus the conformal class of Levi form is determined by $T^{1,0}$.
We assume that $T^{1,0}$ is {\em strictly pseudoconvex} in the sense that the Levi form is positive definite for a choice of $\th$; such a $\th$ is called a {\em pseudo-hermitian structure}, or a {\em contact form}, which is always assumed to be positive in this sense. 

To define the ambient metric, we further need to assume that $M$ is embedded in a complex manifold $X$ of dimension $n+1$, i.e., 
$T^{1,0}$ is given by $\bC TM\cap T^{1,0}X$, where $T^{1,0}X$ is the holomorphic tangent bundle of $X$.  In this situation, a defining function $\rho\in C^\infty(X)$ of $M$, which is positive on the pseudoconvex side, gives a contact form 
\begin{equation}\label{normalrho}
\th=\frac{i}{2}(\pa-\conj\pa)\rho|_{TM}.
\end{equation}
Conversely, if $\th$ is a contact form of $M$, we may find a defining function $\rho$ that satisfies \eqref{normalrho}; such a $\rho$ is determined to the first jet along $M$ and we say that {\em $\rho$ is normalized by $\th$}.  Note that the embeddability follows from the integrability if $M$ is compact and $n\ge2$ (\cite{Bo}, \cite{HL}); moreover, \cite[Th.\ 8.1]{Le} shows that $X$ is taken to be projective algebraic manifold and $M$ is realized as the boundary of a strictly pseudoconvex domain.
In the case $n=1$, it holds only when $M$ is embedded in $\bC^N$ for some $N$
and is not always the case. (We will treat the $n=1$ case independently in \S \ref{three-dim} without assuming the embeddability.)

To motivate the definition of the ambient metric, we first consider the model case 
$M_0\subset\bP^{n+1}$ defined by the quadric 
$$
q(\zeta,\conj\zeta)=\zeta^0\conj\zeta^{n+1}+\zeta^{n+1}\conj\zeta^{0}-\sum_{j=1}^n\zeta^j\conj\zeta{}^j=0
$$
with respect to the homogeneous coordinates $[\zeta^0:\zeta^1:\dots:\zeta^{n+1}]$.
We can identify the CR manifold $M_0$ with the sphere in $S^{2n+1}\subset\bC^{n+1}$, or
the compactification of the hyperquadric in $\bC^{n+1}$.
The {\em ambient space} of $M_0$ is defined to be the tautological bundle  with the zero section removed $\calO(-1)^*=\bC^{n+2}\setminus\{0\}$. We set 
$$
\calN=\{\zeta\in\bC^{n+2}\setminus\{0\}:q(\zeta,\conj\zeta)=0\},
$$
which is the restriction of $\calO(-1)^*$ over $M_0$.
The {\em ambient metric} is the flat Lorentz K\"ahler metric 
$$
\wt g=-i\pa\conj\pa q.
$$
(Here and in the following, we will not make any distinction between K\"ahler metric and form; in our convention, the K\"ahler form $\omega_{\wt g}=-i\pa\conj\pa q$ corresponds to the metric tensor
$\wt g=\wt g_{I\conj J}(d\zeta^I\otimes d\conj\zeta^J+d\conj\zeta^J\otimes d\zeta^I)$ with $\wt g_{I\conj J}=-\pa_I\pa_{\conj J}q$.)  The special unitary group $\SU(n+1,1)$ for the hermitian form $-q(\zeta,\conj\zeta)$ acts on $(\bC^{n+2}\setminus\{0\},\wt g)$ as holomorphic isometries that preserves $\calN$, thus induces CR automorphisms of $M_0$. Moreover we see that the CR automorphism group of $M_0$ is $\SU(n+1,1)/\bZ_{n+2}$.

For a general embedded CR manifold $M\subset X$, we define the ambient space as a fractional power of the canonical bundle with the zero section removed.  We assume that, near $M$ in $X$, there is a bundle $\calL_X$ satisfying $\calL^{n+2}_X=\calK_X$. (Such a bundle $\calL_X$ exists locally and it is sufficient for our purpose; see Remark \ref{canonicalbundlerem}.) The {\em ambient space} of $M$ is defined to be the total space of the $\bC^*$-bundle $\calL_X^*=\calL_X\setminus0$. The restriction of $\calL_X^*$ over $M$ is denoted by $\calN$.  If $\rho$ is a defining function of $M$, then its pullback to $\calL_X^*$, which is also denoted by $\rho$, is a defining function of $\calN$.  Note that $\calN$ is a CR manifold of dimension $2n+3$, the Levi form of which is positive except for the fiber direction.

For $\lambda\in \bC^*$, we define the {\em dilation}
$\delta_\lambda\colon\calL_X^*\to\calL_X^*$ by scalar multiplication
$\delta_\lambda(\xi)=\lambda\xi$, and the space of functions of
{\em homogeneous degree }$(w,w)$, for $w\in\bR$, by
$$
\wt\calE(w)=\{f\in C^\infty(\calL^*_X,\bC): \delta_\lambda^*f=
|\lambda|^{2w}f \quad\text{for any}\ \lambda\in\bC^*\}.
$$
We often consider homogeneous functions defined only on $\rho\ge0$ and smooth up to the boundary $\rho=0$; such functions are also considered as elements of $\wt\calE(w)$.
In case $w=0$, $\wf\in\wt\calE(0)$ is constant on each fiber and
we will identify $\wf$ with a  function 
on $X$.

In the case $X=\bP^{n+1}$, we have $q\in\wt\calE(1)$, which is a natural defining function of $\calN$ such that $-i\pa\conj\pa q$ is flat.  For general CR manifold, we cannot hope to get a flat metric, but any defining function $\rs\in\wt\calE(1)$ of $\calN\subset\calL^*_X$ which is positive on the pseudoconvex side gives a Lorentz K\"ahler metric
$$
\wt g[\rs]=-i\pa\conj\pa\rs
$$
in a neighborhood of $\calN$ in $\calL_X^*$.  Let $\Ric(\wt g[\rs])$, or simply $\Ric[\rs]$, be the Ricci tensor of $\wt g[\rs]$. By following Fefferman \cite{F1}, we normalize $\rs$ by imposing the Einstein equation. 

\begin{prop}\label{Ricci-prop}
There exists a defining function $\rs\in\wt\calE(1)$ of $\calN$ that is positive on the pseudoconvex side and satisfies
\begin{equation}\label{Ric-eq}
\Ric[\rs]=O_+(\rho^n).
\end{equation}
Here $O_+(\rho^m)$ stands for a term of the form $\pa\conj\pa(\rho^{m+2}\phi)$
for a function $\phi\in\wt\calE(0)$.  Moreover, such an $\rs$ is unique modulo $O(\rho^{n+3})$ and, if one writes
\begin{equation}\label{Ric-obs}
\Ric[\rs]=i\eta\,\rs^{n}\pa\rs\wedge\conj\pa\rs+O(\rho^{n+1})
\end{equation}
with $\eta\in\wt\calE(-n-2)$, then $\eta|_\calN$ is determined by the CR structure.

\end{prop}

\begin{defn}
The {\em ambient metric} of $M$ is defined to be $\wt g[\rs]$ with respect to $\rs$ satisfying \eqref{Ric-eq}.
\end{defn}

Since $\eta|_\calN$ is a CR invariant, we cannot refine the error in the equation \eqref{Ric-eq}. So $\eta|_\calN$ is called
the {\em obstruction function}.

\begin{proof}
We first write the equation in terms of local coordinates $(z^J)=(z^0,z)$ of $\calL^*_X$, where $z=(z^j)=(z^1,\dots, z^{n+1})$ is a holomorphic coordinate system of $X$
and the fiber coordinate $z^0$ is defined with respect to the local section
$(dz^1\wedge\cdots\wedge dz^{n+1})^{1/(n+2)}$.  
Then we may write $\rs=|z^0|^2r(z)$ and we have
\begin{equation}\label{RicJeq}
\Ric\big[|z^0|^2 r\big]
=-i\pa\conj\pa\log \det\left(\frac{ \pa^2(-|z^0|^2 r)}{\pa z^J\pa\conj z^K}\right)\\
=-i\pa\conj\pa\log \calJ_z[r],
\end{equation}
where $\calJ_z$ be the complex Monge-Amp\`ere operator 
$$
\calJ_z[r]=(-1)^{n+2}\det\begin{pmatrix}r & \pa_j r\\ \pa_{\conj k}r & \pa_{j\conj k}r
\end{pmatrix}.
$$
Recall from \cite{F1} that there is a defining function $r$ such that
\begin{equation}\label{JeqC}
\calJ_z[r]=1+\wt\eta\,\rs^{n+2}
\end{equation}
for an $\wt\eta\in\wt\calE(-n-2)$.
The proof also shows that if $r$ satisfies \eqref{JeqC}, then
so does
$$
r+\psi\rho^{n+3},\quad  \psi\in C^\infty(\bC^{n+1}),
$$
and these defining functions give all solutions to \eqref{JeqC}.
Moreover, $\wt\eta|_\calN$ is independent of the choice of $r$.
In particular, we see that $\rs$ satisfying \eqref{Ric-eq} exists locally.

We shall show that this construction of $\rs$ is independent of the choice of coordinates. 
Let $(w^0,w)$ be another coordinates of $\calL_X^*$  such that $w=\Phi(z)$ and $w^0=z^0\varphi(z)$, where $\Phi$ is biholomorphic and $\varphi$ is a nonvanishing holomorphic function of $z$.
We write 
$$
(w^0,w)=\Phi_\sharp(z^0,z)=(z^0\varphi(z),\Phi(z)).
$$
Then the chain rule gives
$$
\left|\det\Phi_\sharp'\right|^2
 \det\left(\frac{ \pa^2\rs}{\pa w^J\pa\conj w^K}\right)
  = \det\left(\frac{ \pa^2\rs}{\pa z^J\pa\conj z^K}\right).
$$
Thus, using $\det\Phi_\sharp'=\varphi\det\Phi'$, we have
\begin{equation}\label{trans-J}
\left|\det\Phi'\right|^2|\varphi|^{2(n+2)}\calJ_w[\wh r\,]
=\calJ_z[r],
\end{equation}
where $\wh r(w)=|\varphi(z)|^{-2}r(z)$.
In particular, setting $\varphi(z)=(\det\Phi')^{-1/(n+2)}$, we have
\begin{equation}
\calJ_w[r]
=\calJ_z\big[\left|\varphi\right|^{2}r\big].
\end{equation}
This shows that $\rs$ defined with respect to $z$ and $w$ agree modulo $O(\rho^{n+3})$.
Therefore, by using a partition of the unity, we can define $\rs\in\wt\calE(1)$ 
globally on $\calL^*_X$ near $\calN$ so that 
\eqref{Ric-eq} holds.

We next prove the uniqueness. Suppose that $\rs$ satisfies \eqref{Ric-eq}.
Then,
in the coordinates $(z^0,z)$, we have $\pa\conj\pa\log \calJ_z[r]=\pa\conj\pa(\phi\rho^{n+2})$, and thus 
$\log \calJ_z[r]-\phi\rho^{n+2}$ is pluriharmonic.
Taking a holomorphic function $f$ such that
$$
\log \calJ_z[r]=2\Re f+O(\rho^{n+2}),
$$
we may write
\begin{equation}\label{efJ}
\calJ_z[r]=|e^f|^2\big(1+O(\rho^{n+2})\big).
\end{equation}
Choosing coordinates $w$ such that $\det(\pa w^j/\pa z^k)=e^{f}$, we see from \eqref{trans-J} with $\varphi=1$ that
$$
\calJ_w[r]=1+O(\rho^{n+2}).
$$
The uniqueness of the solution to this equation implies that of $\rs$.
Finally, the equation \eqref{Ric-obs} is obtained by substituting 
\eqref{JeqC} into \eqref{RicJeq}.
\end{proof}

\begin{rem}\label{canonicalbundlerem}\rm
For a general CR manifold $M\subset X$, the bundle $\calL_X^*$ may exist only locally.
However, we can alternatively define the ambient metric on $\calK_X^*$
so that
$$
\calL_X^*\in\eta\mapsto\eta^{n+2}\in\calK_X^*
$$
is an isometry for each locally defined $(\calL_X^*, \wt g)$;
this is the original definition of the ambient metric in \cite{F1}.
The map corresponds to the quotient $\SU(n+1,1)\to\SU(n+1,1)/\bZ_{n+2}$
and the global existence of $\calL^*_X$ is considered as an analogy of Spin structure in Riemannian geometry. The use of $\calL_X^*$ is essential in arguments using the structure group and \cite{F2} used this formulation; see also \cite{CG} on the relation to the Cartan connection. 
One can easily see if the local invariants constructed via $\calL_X^*$ can be patched up to a global one, by checking if the same method works for $\calK^*_X$.
This is always the case for the arguments in this paper and we can assume the existence of $\calL_X^*$ without losing generality.
\end{rem}

\subsection{Complete Einstein-K\"ahler metrics}
We next consider a K\"ahler metric on the strictly pseudoconvex domain $D$ bounded by $M$.  Since
$\rs\in\wt\calE(1)$ defines a hermitian metric of the line bundle $\calL_D$,  the curvature gives a $(1,1)$-form
$$
g=-i\pa\conj\pa\log \rs,
$$
which is a K\"ahler metric on $D$ near $M$. We can also write $g$ as the curvature of
the canonical bundle $\calK_D$ with the hermitian metric $h=\rs^{n+2}$:
$$
g=\frac{-1}{n+2}i\pa\conj\pa\log h.
$$
In local coordinates, we have
$
\calJ_z[r]
=r^{n+2}\det(g_{j\conj k})
$ 
and thus
$$
\Ric(g)+(n+2)g=-i\pa\conj\pa\log \calJ_z[r].
$$
Therefore $\calJ_z[r]=1+O(\rho^{n+2})$ implies
\begin{equation}\label{appeinstein}
\Ric(g)+(n+2)g= O_+(\rho^{n}).
\end{equation}

\begin{defn} A K\"ahler manifold $(D,g)$ is 
{\em asymptotically-Einstein} if  \eqref{appeinstein} holds and
$
g+i\pa\conj\pa\log\rho
$
is $C^\infty$ up to the boundary for a $C^\infty$ defining function 
$\rho$ of $M$.
\end{defn} 

We can apply the theorem of Cheng-Yau \cite{ChY} and its refined from by van Coevering \cite{vC}
to give examples of asymptotically-Einstein K\"ahler manifolds.

\begin{thm}\label{CYthm}
Let $D$ be a bounded strictly pseudoconvex domain in a K\"ahler manifold $X$.
Then $D$ admits a complete Einstein-K\"ahler metric if and only if the canonical bundle $\calK_D$ is positive. Moreover, such a domain $D$ admits an asymptotically Einstein K\"ahler metric.
\end{thm}
\begin{proof}
The first statement is mostly due to \cite{ChY} and this sharp statement is proved in  \cite[Th.\ 3.1]{vC}.
The second statement follows from the asymptotic analysis of $g$ in \cite{LM}.
We recall it for the reader's convenience. We take $\rs$ as in Proposition \ref{Ricci-prop}; then $g$ can be written in the form
$$
g=-i\pa\conj\pa\log\rs+i\pa\conj\pa F,
$$
where $F\in C^\infty(D)$ admits expansion
$F=\rho^{n+2}(\eta_0+\eta_1\log\rho)$ with $\eta_0\in C^\infty(\conj D)$ 
and $\eta_1\in C^{n+1}(\conj D)$. 
Let $\psi(t)$ be a $C^\infty$ function such that $\psi(t)=0$ if $t<1/2$ and $\psi(t)=1$ if $t\ge1$. Then $\psi_\epsilon(z)=\psi(\rho(z)/\epsilon)$ satisfies $|\pa\psi_\epsilon|\le C \rho^{-1}$ and $|\pa\conj\pa\psi_\epsilon|\le C\rho^{-2}$ on $D$ with respect to a metric on $X$. It follows that
$|\pa\conj\pa(\psi_\epsilon F)|\le C\rho$. Taking $\epsilon>0$ small, we may cut off $F$ near the boundary and make the desired K\"ahler metric.
\end{proof}

This theorem, in particular, shows that a strictly pseudoconvex domain in a Stein manifold admits an asymptotically Einstein metric.  Further examples, which are not Stein, include the unit disk bundle in a negative line bundle over a compact K\"ahler manifold. See \cite{vC} for more examples and detailed discussions.

\subsection{Pseudo-hermitian geometry}
Let $\th$ be a contact form on a strictly pseudoconvex CR manifold
$M$. Then there is a uniquely determined real vector field $T$, called the {\em Reeb vector field}, that satisfies
$$
T\CR d\th=0,\quad \th(T)=1.
$$
With this choice, we have a decomposition $\bC TM=\bC T\oplus T^{1,0}\oplus T^{0,1}$.
Let us take a frame $Z_\alpha$ of $T^{1,0}$ and set  $Z_{\conj\beta}=\conj{Z_\beta}$.  Then the set
$$
T,\  Z_\alpha, \ Z_{\conj\beta}
$$
forms a frame of $\bC TM$.  
The dual frame $\th,\th^\alpha,\th^{\conj\beta}$ is said to be an {\em admissible coframe} and satisfies
$$
 d\th=i h_\ab\th^\alpha\wedge\th^{\conj\beta}
$$
for a positive hermitian matrix $h_{\alpha\conj\beta}$.
We will use abstract index notation
and denote $T^{1,0}$ by $\calE^\alpha$ and its dual by $\calE_\alpha$.
Tensor bundles are defined, e.g., by $\calE_{\alpha\conj\beta}:=\calE_\alpha\otimes\calE_{\conj\beta}$, where $\calE_{\conj\beta}$ is the dual of $T^{0,1}$.

The {\em canonical bundle} $\calK_M$ of $M$ is defined by
$\wedge^{n+1} (T^{0,1})^\perp\subset\wedge^{n+1}\bC T^*M$.
If $M$ is embedded in $X$, we can identify $\calK_M$ with the restriction of 
$\calK_X$ over $M$.  
Given a contact form $\th$, we can find a local section $\zeta$ of the canonical bundle $\calK_M$ such that
\begin{equation}\label{th-normal}
\th\wedge d\th^n=i^{n^2}n!\,\th\wedge(T\CR\zeta)\wedge(T\CR\conj\zeta).
\end{equation}
We then say that $\th$ {\em is volume-normalized with respect to} $\zeta$. 
We take a (locally defined) bundle $\calL_M=\calK_M^{1/(n+2)}$.
A {\em CR density of weight $w\in\bR$} (or $(w,w)$ in the formulation of \cite{GG}) is a $C^\infty$ section of the bundle
$$
\calE(w)=\calL_M^{-w}\otimes\conj\calL_M^{\,\,-w},
$$ 
which is defined globally even if $\calL_M$ is not.  We abuse the notation and also denote the space of the sections of the bundle by
$\calE(w)$. We can also identify a section $\varphi$ of $\calE(w)$ with a function $f$ on 
$\calN=\calL_M\setminus 0$ which is homogeneous of degree $(w,w)$
by the correspondence $\varphi(x)=f(\xi)|\xi|^{-2w}$. Here $|\xi|^2=\xi\otimes\conj\xi\in\calE(-1)$ for $\xi\in\calL_M$ and $\xi$ projects to $x\in M$.
With a fixed choice of $\th$, we can also identify a CR density $f$ with a function on $M$.
Choose a section $\zeta$ that volume-normalize $\th$ and set $\eta=\zeta^{1/(n+2)}$,
which is a local section of $\calN$.  Then the function
$f_\th=f\circ\eta$
is globally defined on $M$ and satisfies the transformation law
under the scaling $\wh\th=e^\up\th$:
$$
f_{\wh\th}=e^{w\up}f_\th.
$$
In particular, $\calE:=\calE(0)$ is the space of functions on $M$ and
$\calE(-n-1)$ can be identified with the space of volume forms by the correspondence
$$
\calE(-n-1)\ni f\ \longleftrightarrow \ f_\th \th\wedge d\th^n\in \wedge^{2n+1}T^*M.
$$
Note also that $\eta$ as above defines a density  $|\eta|^{-2}\in\calE(1)$
and $\bth=\th\otimes |\eta|^{-2}$ gives a canonical section of $T^*M\otimes\calE(1)$,
i.e., $\bth$ is independent of $\th$.  
Alternatively, we can define $\bth$ as a 1-form on $\calN$:
$$
\bth=\frac{i}{2}(\pa-\conj\pa)\rs|_{T\calN}.
$$
The Levi form of $\th$ scales the same way and defines
a canonical section $\bh_{\alpha\conj\beta}\in\calE_{\alpha\conj\beta}(1)$;
here we simplify the notation by setting
$\calE_{\alpha\conj\beta}(1):=\calE_{\alpha\conj\beta}\otimes\calE(1)$.
We will use $\bh_{\alpha\conj\beta}$  and its inverse $\bh^{\alpha\conj\beta}\in
\calE^{\alpha\conj\beta}(-1)$ 
to raise and lower the indices. These operations are independent of $\th$.

For a homogeneous function
$\wf\in\wt\calE(w)$ on the ambient space, its restriction to $\calN$ defines a CR density
 $f=\wf|_\calN\in\calE(w)$.
We call $\wf$ an {\em ambient extension of} $f$, which has ambiguity of adding
$\varphi\,\rs $ with $\varphi\in\wt\calE(w-1)$.
The following lemma will help to see the correspondence between CR densities
and their extensions.

\begin{lem}\label{divlem}
{\rm (i)} A defining function $\rho$ satisfies $\calJ_z[\rho]=1+O(\rho)$ if and only if
 $\rho$ is normalized by $\th$ that is volume-normalized with respect to 
$dz=dz^1\wedge\cdots\wedge dz^{n+1}$.

\medskip
\noindent
{\rm (ii)}
If $\rho$ is normalized by $\th$,
then, for each 
 $F\in\wt\calE(-m)$, there exists an  $F'\in\wt\calE(0)$  such that
$$
F\rs^m=F' \rho^m
\quad\text{and}\quad 
F'|_M=(F|_\calN)_\th.
$$
\end{lem}

\begin{proof}
(i) It is clear from the equation by \cite{Fa}: 
\begin{equation*}\label{FarrisJ}
\th\wedge d\th^n=i^{n^2}n!\,\calJ_z[\rho]\th\wedge(T\CR dz)\wedge(T\CR d\conj z).
\end{equation*}
(ii)
We first take $\rho$ as in (i) and fix the fiber coordinate $z^0$ by the section
$(dz)^{1/(n+2)}$.
Then we may write $F=|z^0|^{-2m}F'(z)$ and $\rs=|z^0|^2\rho(z)$; thus 
$
(F|_\calN)_\th=F'|_M
$.
For general $\wh\th=e^\up \th$, extending $\up$ to
$\wt\calE(0)$, we obtain $\wh\rho=e^\up\rho$ which is normalized by $\wh\th$.
So $F'\rho^m=\wh F'\,\wh\rho\,{}^m$ gives $\wh F'=e^{-m\up}F'$ and
hence $\wh F'|_M=(F|_\calN)_{\wh\th}$.
\end{proof}

We next recall the canonical connection for the pseudo-hermitian geometry.
A choice of $\th$ determines 
the {\em Tanaka-Webster connection} $\nabla$ on $T^{1,0}$:
the connection form $\omega_\alpha{}^\beta$, $\nabla Z_\alpha=\omega_\alpha{}^\beta Z_\beta$, is uniquely determined by the structure equations

$$
d\th^\beta=\th^\alpha\wedge\omega_\alpha{}^\beta+A^\alpha{}_{\conj\beta}\,\bth\wedge\th^{\conj\beta},\quad
\omega_{\alpha\conj\beta}+\conj{\omega_{\beta\conj\alpha}}=d\bh_{\alpha\conj\beta},
\quad A_{[\alpha\beta]}=0.
$$
Here $[\cdots]$ indicates antisymmetrisation over the enclosed indices:
$$A_{[\alpha\beta]}=\frac{1}{2}(A_{\alpha\beta}-A_{\beta\alpha}).
$$
Thus $A_{\alpha\beta}\in\calE_{\alpha\beta}$ is symmetric and is called the {\em Tanka-Webster torsion}.
We extend the connection to the one on $\bC TM$ by imposing $\nabla T=0$
and extending to $T^{0,1}$ by conjugation.  There is an induced connection on the canonical bundle $\calK_M$ and also on the density bundles $\calE(w)$.

The curvature of the connection is defined by
$$
d\omega_\alpha{}^\beta-\omega_\alpha{}^\gamma\wedge\omega_\gamma{}^\beta=
R_\alpha{}^\beta{}_{\rho\conj\sigma}\th^\rho\wedge\th^{\conj\sigma}\mod
\th,\th^\alpha\wedge\th^\beta,\th^{\conj \alpha}\wedge\th^{\conj\beta}.
$$
We call $R_{\alpha\conj\beta\gamma\conj\sigma}\in\calE_{\alpha\conj\beta\gamma\conj\sigma}(1)$ the {\em Tanaka--Webster curvature}. Other components of
the curvature form can be written in terms of $R_{\alpha\conj\beta\gamma\conj\sigma}$, $A_{\alpha\beta}$, $A_{\conj\alpha\conj\beta}=\conj{A_{\alpha\beta}}$ and their covariant derivatives.  
The Ricci tensor and the scalar curvature are defined respectively by
$$
\Ric_{\alpha\conj\beta}=R_{\gamma}{}^{\gamma}{}_{\alpha\conj\beta}\in
\calE_{\alpha\conj\beta},
\quad
\Scal=\Ric_{\alpha}{}^\alpha\in\calE(-1).
$$
We denote the components of successive covariant derivatives of a tensor by subscripts preceded by a comma, 
as in $A_{\alpha\beta,\gamma\conj\sigma}$.  When the derivatives are applied to a function, we omit the comma.  With these notations, we set
$$
\pa_bf=f_\alpha \th^\alpha,\quad\conj\pa_bf=f_{\conj\alpha} \th^{\conj\alpha},\quad 
\Delta_bf=-f_\alpha{}^\alpha-f^\alpha{}_\alpha.
$$
Note that the sub-Laplacian changes the weight 
$
\Delta_b\colon\calE(w)\to\calE(w-1).
$
We will use the index $0$ to denote the $\th$ component, so that $f_0=Tf$ for a density, where $T$ is regarded as an operator $T\colon\calE(w)\to\calE(w-1)$. Then the commutator of the derivatives on $f\in\calE(w)$ 
are given by
\begin{equation}\label{commf}
2f_{[\alpha\beta]}=0,\qquad
2f_{[\alpha\conj\beta]}=i\bh_{\alpha\conj\beta}f_0,\qquad
2f_{[0\alpha]}=A_{\alpha\beta} f^\beta
\end{equation}
and the Bianchi identities give
\begin{equation}\label{BianchiA}
A_{\alpha[\beta,\gamma]}=0,
\quad
A_{\alpha\beta},{}^{\alpha\beta}+A_{\conj\alpha\conj\beta},{}^{\conj\alpha\conj\beta}=\Scal_{0}.
\end{equation}
See \cite{L1,L2} for a complete list of such formulas and the proof.

We will use the following transformation rules of connection under the scaling of contact form. Let $\wh\th=e^\up\th$ and take $\wh\th^\alpha=\th^\alpha+\up^\alpha\th$ as an admissible coframe for $\wh\th$,
 where $\up_\alpha=\nabla_\alpha\up$.
We denote the quantities defined with respect to $\wh\th$, and the components in the coframe $\wh\th^\alpha$, by $\wh\nabla_\alpha$,  $\wh A_{\alpha\beta}$, etc.

\begin{prop}\label{transf-prop}
{\rm (i)}
The torsion $A_{\alpha\beta}\in\calE_{\alpha\beta}$ and
the scalar curvature $\Scal\in\calE(-1)$ satisfy
$$
\begin{aligned}
\wh A_{\alpha\beta}&=A_{\alpha\beta}+i \up_{\alpha\beta}-i\up_\alpha \up_\beta,
\\
\wh \Scal&=\Scal+(n+1)\Delta_b \up-n(n+1) \up_\alpha \up^\alpha.
\end{aligned}
$$
{\rm (ii)}
If $f\in\calE(w)$, then
$$
\begin{aligned}
\wh\nabla_\alpha f&=f_\alpha +w\up_\alpha f,\\
\wh\nabla_{0}f&=f_{0}+i\up_{\alpha}f^\alpha
-i\up^{\alpha}f_\alpha
+\frac{2w}{n+2}(\up_0+\Im\up_\alpha{}^\alpha)f.
\end{aligned}
$$
{\rm (iii)}
If $\tau_\alpha\in\calE_\alpha(w)$, then
$$
\wh\nabla_{\conj\beta}\tau_\alpha =
\nabla_{\conj\beta}\tau_\alpha+\bh_{\alpha\conj\beta}\up^\gamma\tau_\gamma+
w\up_{\conj\beta}\,\tau_\alpha.
$$
{\rm (iv)} If $f\in\calE(w)$, then
\begin{equation*}
\begin{aligned}
\wh\Delta_b f=\Delta_b f+& w(\Delta_b\up) f-(n+2w)(\up^\alpha f_\alpha+\up_\alpha f^\alpha)\\
&-2w(n+w)\up^\alpha\up_\alpha f.
\end{aligned}
\end{equation*}
\end{prop}

The proofs of (i), (ii) and (iii) can be found in \cite{L1,L2} and \cite{GG};
(iv) is an easy consequence of (ii) and (iii).

\section{CR invariant differential operators and \\ CR pluriharmonic functions}
\subsection{GJMS operators}
We shall apply the ambient metric to construct CR invariant differential operators
by following \cite{GJMS} and \cite{GG}.
We here use abstract index notation and denote the $(1,0)$-form $-\pa\rs$ by $Z_I$ and its conjugate by $Z_{\conj I}$. We
use $\wt g^{I\conj J}$ to raise the index, e.g., $Z^I=\wt g^{I\conj J}Z_{\conj J}$.
The covariant derivative of type $(1,0)$ is denoted by $\wt\nabla_I$ and that of type $(0,1)$ by $\wt\nabla_{\conj I}$.
Then we have 
\begin{equation}\label{DZeq}
\wt\nabla_I\rs=-Z_I,
\quad \wt\nabla_I Z_{\conj J}=\wt g_{I\conj J},\quad Z^I Z_I=-\rs
\end{equation}
and these relations can be used to compute the commutators of $\rs$
and the Laplacian 
$\wt\Delta=-\wt\nabla_I\wt\nabla^I$ acting on functions on $\calL_X^*$:
\begin{equation}\label{commutator}
\begin{aligned}[]
[\wt\Delta,\rs^\ell]&=\ell\,\rs^{\ell-1}(Z+\conj Z+n+\ell+1),\\ 
[\wt\Delta^\ell,\rs]&=\ell\,(Z+\conj Z+n+\ell+1)\wt\Delta^{\ell-1},
\end{aligned}
\end{equation}
where  $Z=Z^I\wt\nabla_I$, $\conj Z=Z^{\conj I}\wt\nabla_{\conj I}$.  Note that $Z$ and $\conj Z$ act as scalar multiplications on $f\in\wt\calE(w)$:
$$
Zf=\conj Z f=wf.
$$

We next recall the relation between $\wt\Delta$ and the Laplacian of $g$ on the domain $D$.
 We identify $u\in C^\infty(D)$
and its lift in $\wt\calE(0)$ over $D$. 
Then, for $w\in\bR$, we have
\begin{equation}
\rs^{1-w}\wt\Delta(\rs^w u)=(\Delta+w(n+1+w)) u,
\end{equation} 
where $\Delta=-\nabla_j\nabla^j$ is the K\"ahler Laplacian of $g$; see \cite[Prop.\ 5.4]{GG}.
In particular, if $w=0$, then $\rs \wt\Delta u=\Delta u$.

For $f\in\calE(m)$, take its ambient extension $\wf\in\wt\calE(m)$.
Then, for $2m\in [-n,0]\cap\bZ$,
$$
\wt\Delta^{n+2m+1}\wf|_\calN\in\calE(-n-m-1)
$$
depends only on $f$.  In fact, if $\wf=\rs\varphi$
for $\varphi\in\wt\calE(m-1)$,
\begin{equation}\label{{GJMS-comp}}
\begin{aligned}
\wt\Delta^{n+2m+1}\rs\varphi&=[\wt\Delta^{n+2m+1},\rs]\varphi+O(\rho)\\
&=(n+2m+1)(Z+\conj Z+2n+2m+2)\wt\Delta^{n+2m}\varphi+O(\rho)\\
&=O(\rho).
\end{aligned}
\end{equation}
Therefore $P_{n+2m+1}f=\wt\Delta^{n+2m+1}\wf|_\calN$ gives a differential operator
$$
P_{n+2m+1}\colon\calE(m)\to\calE(-n-m-1),
$$
which is called a {\em GJMS operator} \cite{GJMS}; this construction works for more general class of CR densities $\calE(w,w')$, see \cite{GG}.
The upper bound on $m$ is imposed to ensure that $P_{n+2m+1}$ is independent of the ambiguity in $\rs$. 

In the following, we only use the case $m=0$ and set $P=P_{n+1}$, which acts on $\calE$. As in the conformal case,
we can also characterize $P$ as the compatibility operator for harmonic extension.

\begin{lem}\label{CRGJMSlem}
For each $f\in  \calE$, there exist functions $A\in\wt\calE(0)$ and $B\in\wt\calE(-n-1)$ 
such that
\begin{equation}\label{wtdeltaAB}
\wt\Delta(A+B\rs^{n+1}\log \rho)=O(\rho^\infty)
\quad\text{and}
\quad A|_\calN=f.
\end{equation}
Moreover, $B|_\calN=\frac{(-1)^{n+1}}{(n+1)!n!}Pf$ holds.
\end{lem}

\begin{proof}
We prove the lemma by giving an inductive step to construct $A$ and $B$.
This procedure will be used repeatedly in this paper.  

We construct $A_k\in\wt\calE(0)$, $k\le n$, such that
\begin{equation}\label{appAk}
\wt\Delta A_k=\rs^k\phi_k \text{ for a } \phi_k\in\wt\calE(-k-1)
\quad\text{ and}\quad A_k|_\calN=f.
\end{equation}
By taking an extension of $f$ to $\wf\in\wt\calE(0)$, we may set $A_0=\wf$ 
and $\phi_0=\wt\Delta\wf$.
If $A_k$ is given, we set $A_{k+1}=A_k+\rs^{k+1}\psi_{k+1}$ with $\psi_{k+1}\in\wt\calE(-k-1)$ and compute
$$
\begin{aligned}
\wt\Delta A_{k+1}&=\wt\Delta A_k+[\wt\Delta, \rs^{k+1}]\psi_{k+1}+\rs^{k+1}\wt\Delta\psi_{k+1}\\
&=\rs^{k}\phi_k+(k+1)(n-k)\rs^{k}\psi_{k+1}+O(\rho^{k+1}).
\end{aligned}
$$
Then we obtain $A_{k+1}$ satisfying \eqref{appAk} by setting $\psi_{k+1}=\frac{-1}{(k+1)(n-k)}\phi_k$ and $\phi_{k+1}=\wt\Delta\psi_{k+1}$. It follows that $\phi_{n}={(-1)^n}(n!)^{-2}\wt\Delta^{n+1}\wf$.
This construction breaks down exactly when $k= n$, and at this step,
by using
$$
\rs\wt\Delta\log\rho=\Delta\log\rho=(n+1)+O(\rho),
$$
we have
$$
\begin{aligned}
\wt\Delta(A_{n}+B_0\rs^{n+1}\log\rho)
=\phi_n\rs^n-(n+1)B_0\rs^{n}+O(\rho^{n+1}\log\rho).
\end{aligned}
$$
Therefore, setting $B_0=\frac{(-1)^{n+1}}{(n+1)!n!}\wt\Delta^{n+1}\wf$, we can continue the inductive step to determine $A$ and $B$.
\end{proof}

From this proof we can also see that \eqref{appAk} determines $A_k$ mod $O(\rho^{k+1})$ for $k\le n$ and that $B$ mod $O(\rho^\infty)$ is independent of the choice of $\rho$, which is used to define the singularity $\log \rho$.

We shall reformulate \eqref{wtdeltaAB} on the complete manifold $(D,g)$.
Since $\rs\wt\Delta=\Delta$ on $\wt\calE(0)$, we can write \eqref{wtdeltaAB} as
\begin{equation}\label{deltaAB}
\Delta(A+B'\rho^{n+1}\log \rho)=O(\rho^\infty),
\quad A|_M=f,
\end{equation}
where
$A,B'\in C^\infty(\conj D)$. Then, in view of Lemma \ref{divlem}, 
we have
$$
B'|_M=\frac{(-1)^n}{(n+1)!n!} (Pf)_\th\quad\text{with }\ 
\th=\frac{i}2(\pa-\conj\pa)\rho|_{TM}.
$$

\subsection{CR pluriharmonic functions}
So far, the ambient metric construction of CR invariant operators is completely parallel to the conformal case.  We here mention one important property that is specific to the CR case.  Since the ambient metric is K\"ahler, $\ker\wt\Delta$ contains pluriharmonic functions, and accordingly $\ker P$ contains the space $\calP$ of CR pluriharmonic functions on $M$. 
Recall that a {\em CR function} is a complex valued $C^\infty$ function $f$ on $M$
such that $\conj\pa_b f=0$ and a {\em CR pluriharmonic function} is a real-valued function on $M$ that is locally the real part of a CR function.
By the strictly pseudoconvexity, we can extend a CR function to a holomorphic function on the pseudoconvex side of $M$; the same is true for a CR pluriharmonic function.

When $M=S^{2n+1}$ it is shown in that 
$
\ker P=\calP.
$
The equality is also confirmed in case $M$ is 3-dimensional and torsion-free \cite{GL}.
See \cite{CCY} for recent progress on the study of this equality.

\begin{lem}\label{Pformop}
Let $M\subset X$ be a strictly pseudoconvex CR manifold and $\iota\colon M\to X$ be 
the inclusion.

{\rm (i)} 
For a real-valued function $f\in\calE$,  take
an extension $\wt f\in\wt\calE(0)$ such that $\wt\Delta\wt f=O(\rho)$. 
Then
$
\bfP (f)=\iota^*(i\pa\conj\pa\wt f\,)
$
depends only on $f$ and defines an operator
$$
\bfP \colon\calE\to C^\infty(M,\wedge^2T^*M).
$$
Moreover, $\ker\bfP =\calP$ holds. 

\medskip

\noindent
{\rm (ii)} 
In terms of the Tanaka-Webster connection, one has
\begin{equation}\label{bPform}
\bfP f=P_{\alpha\conj\beta}f\,\th^\alpha\wedge\th^{\conj\beta}
+P_{\alpha}f\,\th\wedge\th^\alpha+
P_{\conj\beta}f\,\th\wedge\th^{\conj\beta},
\end{equation}
where
$$
P_{\alpha\conj\beta}f=i\tf f_{\alpha\conj\beta}
:=i\left(f_{\alpha\conj\beta}-\frac1n f_\gamma{}^\gamma \bh_{\alpha\conj\beta}\right),
\quad
  P_\alpha f =\frac{1}{n}f_{\conj\beta}{}^{\conj\beta}{}_{\alpha}+iA_{\alpha\beta}f^\beta,
 $$
 and these operators satisfy  
 \begin{equation}\label{Pdiv}
(P_{\alpha\conj\beta}f)_{,}{}^{\conj\beta}=i(n-1)P_\alpha f.
 \end{equation}
 \end{lem}
 \begin{proof}
 (i)
 Since $\wf$ is unique modulo $O(\rho^2)$, the ambiguity causes the difference
 $$
 \pa\conj\pa(\rs^2 \varphi)=\varphi\pa\rs\wedge\conj\pa\rs+O(\rho),
 $$
but $\iota^*(\pa\rs\wedge\conj\pa\rs)=\th\wedge\th=0$. Thus $\bfP f$ is well-defined.

 Let $\eta=\iota^*(i(\conj\pa-\pa)\wt f)$. Then we have $\eta=i(\pa_b-\conj\pa_b)f+\lambda\th$ for a real function $\lambda$, and $\bfP f=d\eta$.  If $\bfP f=0$, then $d\eta=0$. So, locally one may find a real function $h$ such that
$dh=i(\pa_b-\conj\pa_b)f+\lambda\th$.  Restricting this formula to $T^{0,1}$ gives $\conj\pa_b(f-ih)=0$.  Conversely, if $f$ is CR pluriharmonic, then we may take $\wf$ to be pluriharmonic.  Then $\pa\conj\pa\wf=0$, and $\bfP f=0$ follows.  

(ii) Let $f,\eta, \lambda$ be as above.
Since $i\pa\conj\pa\wt f$ is trace-free with respect to $\wt g^{I\conj J}$ on $\calN$, so is $\iota^*(i\pa\conj\pa\wt f)$ with respect to $\bh^{\alpha\conj\beta}$.
It follows that $\tf d\eta=0$, which forces 
 $\eta=i(\pa_b-\conj\pa_b)f+(1/n)\Delta_bf \th$. The expression of $P_{\alpha\conj\beta}$ and $P_\alpha$ are obtained by expanding $d\eta=d(i(\pa_b-\conj\pa_b)f+(1/n)\Delta_bf \th)$. To prove \eqref{Pdiv}, we apply $d$ to \eqref{bPform}. 
Then the type $(2,1)$ part gives the identity
$$
(P_{\alpha\conj\beta} f_{,\gamma}-iP_{\alpha}f\, \bh_{\gamma\conj\beta})\th^\gamma\wedge\th^{\alpha}\wedge\th^{\conj\beta}=0,
$$
or equivalently
$
P_{[\alpha|\conj\beta}f_{,|\gamma]}=iP_{[\alpha}f\, \bh_{\gamma]\conj\beta}.
$
Taking contraction with $\bh^{\gamma\conj\beta}$ and using $P_\alpha{}^\alpha f=0$, we get \eqref{Pdiv}.
\end{proof}

In the case $n=1$, we trivially have $P_{1\conj 1}f=0$ 
and $\bfP f=0$ is reduced to $P_1f=0$, while for
$n>1$, $P_{\alpha\conj\beta}f=0$ forces $P_{\alpha}f=0$ so that
$\bfP f=0$ if and only if  $P_{\alpha\conj\beta}f=0$.
By the construction, $\bfP$ is a CR invariant operator, and so are
$$
\begin{aligned}
P_{\alpha\conj\beta}&\colon\calE\longrightarrow\ \ \calE_{\alpha\conj\beta} \quad &&\text{for }n>1
\\
P_{\alpha}\ &\colon\calE\longrightarrow\calE_{\alpha}(-1) \quad&&\text{for }n=1.
\end{aligned}
$$
It should be worth noting that these operators arise from the compositions of more primitive CR invariant operators
$$
\begin{aligned}
P_{\alpha\conj\beta}&\colon\calE\xrightarrow{\conj\pa_b}\calE_{\conj\beta}\longrightarrow\ \ \calE_{\alpha\conj\beta} \quad &&\text{for }n>1
\\
P_{\alpha}\ &\colon\calE\xrightarrow{\conj\pa_b}\calE_{\conj\beta}\longrightarrow\calE_{\alpha}(-1) \quad&&\text{for }n=1.
\end{aligned}
$$
The second operators can be read from the formulas of $P_{\alpha\conj\beta}$ and $P_{\alpha}$.
These are parts of Bernstein-Gelfand-Gelfand sequence constructed in \cite{CSS}.
This is also a specific feature of CR case.

In the following sections, we use pluriharmonic extensions of $f\in\calP$.
By the strict pseudoconvexity, we can extend $f$ to a pluriharmonic function on a neighborhood $U$ of $M$ in $\conj D$.  If $D$ is Stein, we can further extend $\wf$ to a pluriharmonic function on $D$ as a consequence of Hartogs extension theorem \cite{BF}.
Thus, for such domains, $\calP$ can be identified with the space of pluriharmonic functions on $D$ with $C^\infty$ boundary values.
The extension of CR pluriharmonic functions holds for more general case; see \S7.

\section{$P$-prime operator}

\subsection{Definition of $P$-prime operator}

For a contact form $\th$ on $M\subset X$, we take a normalized defining function $\rho$.
Then we may decompose $g$ into two parts:
$$
g=-i\pa\conj\pa\log \rho-i\pa\conj\pa \log h_\th,
$$
where $h_\th=\rs/\rho$ and $\pa\conj\pa \log h_\th$ is smooth up to the boundary. Let us emphasize the fact that $h_\th$ is defined only mod $O(\rho)$ since $\rho$ has ambiguity of $O(\rho^2)$. We will
further normalize $h_\th$ and fix it mod $O(\rho^2)$ in the next section.
However, we here leave the maximum ambiguity as it causes no effect to the following 
\begin{prop}
Let $h_\th\in\wt\calE(1)$ be as above. For $f\in\calP$,
take its pluriharmonic extension $\wf$ to the pseudoconvex side of $M$. Then
$$
\wt\Delta^{n+1}\big(\wf\log h_\th\big)
\in\wt\calE(-n-1)
$$ and its value on $\calN$ is determined by $\th$ and $f$.
\end{prop}

\begin{defn}\label{def-pp} The
{\em $P$-prime operator} for $\th$ is defined by
$$
P'\colon\calP\to\calE(-n-1),\quad
P'f=-\wt\Delta^{n+1}\big(\wf\log h_\th\big)|_{\calN}
$$
and the
{\em $Q$-curvature} is defined by the constant term of $P'$:
$$
Q=-\wt\Delta^{n+1}\big(\log h_\th\big)|_{\calN}\in\calE(-n-1).
$$
\end{defn}
\begin{proof} 
By the Leibniz rule, we have
\begin{equation}\label{expandwflog}
\wt\Delta (\wf\log h_\th)=\wt\Delta\wf\cdot\log h_\th
-\langle d\wf,d\log h_\th\rangle_{\wt g}+\wt f\,\wt\Delta\log h_\th.
\end{equation}
The first term vanishes since $\pa\conj\pa\wf=0$.
To compute the last two terms, take a fiber coordinate $z^0$ and write
$\log h_\th=\log|z^0|^2+\varphi(z)$, where $\varphi\in\wt\calE(0)$ is smooth up to $\calN$. Then $\pa\log h_\th=dz^0/z^0+\pa\varphi$ and we see that
$$
-\langle d\wf,d\log h_\th\rangle_{\wt g}+\wt f\,\wt\Delta\log h_\th\in\wt\calE(-1).
$$
Therefore $\wt\Delta^{n+1}(\wf\log h_\th)$ is an element of $\wt\calE(-n-1)$.

We next examine the effect of the ambiguity in $h_\th$.
Suppose that $\wh h_\th$ is another choice.
Then
$
\log\wh h_\th-\log h_\th=\rs\varphi
$
for $\varphi\in\wt\calE(-1)$. Thus,  we have
$$
\wt\Delta^{n+1}(\wf\log\wh h_\th-\wf\log h_\th)=\wt\Delta^{n+1}\rs\wf\varphi=O(\rho),
$$
which shows the second claim. 
\end{proof}

The transformation law of $P'$ naturally follows from the definition.

\begin{prop}\label{trans-PP-prop} 
If $\wh \th=e^\up\th$, then
\begin{equation}\label{trans-PP}
\wh P'f=P'f+P(\up f),\quad f\in\calP.
\end{equation}
\end{prop}

\begin{proof}
We extend $\up$ to $\up\in\wt\calE(0)$.
Then we may set $h_{\wh\th} =e^{-\up} h_{\th}$ and so
$$
\log h_{\wh \th}=\log h_\th-\up.
$$
Applying $\wt\Delta^{n+1}$ to this equation, we obtain \eqref{trans-PP}.
\end{proof}

As in the case of $P$-operator, we can characterize $P'$ as an obstruction to the existence of smooth harmonic extension. This formulation will be used in the next subsection to prove the formal self-adjointness of $P'$.

\begin{lem}
Suppose that $\rho$ is normalized by $\th$.
Then, for each pluriharmonic function $f$ on $\conj D$,
there exist  $F, G\in C^\infty(\conj D)$ such that $F=O(\rho)$ and
\begin{equation}\label{ppeq}
\Delta(f\log\rho-F-G \,\rho^{n+1}\log\rho)=(n+1)f+O(\rho^\infty).
\end{equation}
Moreover, $G|_M=\frac{(-1)^{n+1}}{(n+1)!n!}(P'f)_\th$ holds.
\end{lem}

\begin{proof}
By following the proof of Lemma \ref{CRGJMSlem}, we see that
\begin{equation}\label{ppambeq}
\left\{
\begin{aligned}
&\wt\Delta(f\log h_\th+F+\wt G\,\rs^{n+1}\log\rho)=O(\rho^\infty),
\\
& F\in\wt\calE(0),\quad F|_\calN=0,\quad \wt G\in\wt\calE(-n-1)
\end{aligned}
\right.
\end{equation}
admits a solution and $\wt G|_\calN=\frac{(-1)^{n+1}}{(n+1)!n!}P'f$ holds.
On the other hand,
$$
\begin{aligned}
r_\sharp \wt\Delta(f\log h_\th)
&=r_\sharp \wt\Delta(f\log r_\sharp-f\log \rho)\\
&=(\Delta f)\log\rs-Z f-\conj Z f+(n+1)f-\Delta(f\log \rho)\\
&=(n+1)f-\Delta(f\log \rho).
\end{aligned}
$$
Thus \eqref{ppambeq} is reduced to \eqref{ppeq}.
\end{proof}

\subsection{Self-adjointness of $P'$}\label{PPprsec}
Now we are ready to prove one of our main results.

\begin{thm}\label{selfP}
Let $M$ be the boundary of a relatively compact strictly pseudoconvex domain $D$ in a Stein manifold.
Then the operators $P$ and $P'$ defined with respect to a contact form $\th$
on $M$ are formally self-adjoint respectively on $\calE$ and $\calP$, i.e.,
$$
\begin{aligned}
\int_M (f_1 Pf_2- f_2P f_1)&=0\quad \text{for any }f_1,f_2\in\calE
\\
\int_M (f_1 P'f_2- f_2P' f_1)&=0\quad \text{for any }f_1,f_2\in\calP.
\end{aligned}
$$
\end{thm}

The self-adjointness of $P$ has been known for general nondegenerate CR manifolds;
a proof can be found in Gover-Graham \cite{GG}, which reduces the problem
to the conformal GJMS operator via Fefferman's conformal structure. 
We here give a direct proof by using Green's formula for the metric $g$ on $D$.
The argument for $P$ is completely parallel to the conformal (or conformally compact Einstein) case; the one for $P'$ requires an additional observation based
on the pluriharmonic extension of $\calP$.

We first formulate Green's formula for the metric $g$ by following \cite{S}.  It is then convenient to use a defining function $\rho$ normalized by $\th$ and
$
 \|\pa\log\rho\|_g=1
$
near $M$ in $D$. 
With this $\rho$, we define a foliation of $D$ by CR manifolds
$M_\epsilon=\{\rho=\epsilon\}$.
Let $\wt\calH$ be the complex subbundle of $T^{1,0}X$ defined by
$$
\wt T^{1,0}=\{W\in T^{1,0}X:W\rho=0\}.
$$
Then there is a unique type $(1,0)$-vector field $\xi$ such that
$$
\xi\perp_g\wt T^{1,0},\qquad \xi\rho=1.
$$
Let us write $\xi=N+(i/2)T$.  While $g$ diverges at $M$, we
can still show that $\xi$ is continuous up to the boundary $M$ and
$T$ agrees with the Reeb vector field for $\th$ on $M$; see \cite[Lem.\ 3.2]{S}.
The unit outward normal vector field of $M_\epsilon$, for small $\epsilon>0$, is given by
$$
\sqrt{2}\,\nu\quad\text{with}\quad\nu=\rho N.
$$
Since the volume form of $g$ is of the form
$$
dv_g=c_n\big(1+O(\rho)\big)\rho^{-n-2}d\rho\wedge\th\wedge d\th^{n},
\quad c_n=\frac{-1}{n!},
$$
we have
$$
\nu\CR dv_g=\rho^{-n-1}d\sigma,
\quad\text{where}\ \
d\sigma=c_n\big(1+O(\rho)\big)\th\wedge d\th^{n}.
$$
Here $\th$ is extended to $\conj D$ by $\th=\frac{i}{2}(\pa-\conj\pa)\rho$.
Thus  Green's formula for the K\"ahler Laplacian $\Delta=-\nabla^j\nabla_j$ (the half of Riemannian Laplacian) on the subdomain $\rho>\epsilon$ gives
\begin{equation}\label{green}
\int_{\rho>\epsilon} \big(-u\Delta  v+\langle \pa u,\pa v\rangle_g \big)dv_g
=\epsilon^{-n-1}\int_{\rho=\epsilon}u \cdot \nu v\,d\sigma.
\end{equation}
for $u,v\in C^\infty(D)$.

\medskip

\begin{proof}[Proof of Theorem \ref{selfP}.]
Suppose that $u_j=A_j+B_j\rho^{n+1}\log\rho$ is a solution to
$\Delta u_j=0$ with $A_j|_{\rho=0}=f_j$ for $j=1,2$.
Using $\Delta u_2=0$, we have
\begin{equation}\label{Pform}
\int_{\rho>\epsilon} 
\langle \pa u_1,\pa u_2 \rangle dv_{g}
=\epsilon^{-n-1}\int_{\rho=\epsilon}
u_1 \cdot\nu u_2 d\sigma.
\end{equation}
On the other hand,
$$
\begin{aligned}
\rho^{-n-1}u_1\cdot\nu u_2&=(\rho^{-n-1}A_1+B_1\log\rho)\cdot\nu(A_2+B_2\rho^{n+1}\log\rho)\\
&=(\rho^{-n-1}A_1+B_1\log\rho)\\
&\qquad\times(O(\rho)+\big((n+1)B_2\rho^{n+1}+O(\rho^{n+2})\big)\log\rho).
\end{aligned}
$$
The coefficient of $\log\rho$ of the right-hand side is
$
(n+1)A_1B_2+O(\rho).
$
Since $B_2$ is a constant multiple of $(Pf_2)_\th$, we see that
the coefficient of $\epsilon^0\log\epsilon$ of the right-hand side of \eqref{Pform} is
a constant multiple of 
$$
\int_{M}f_1 Pf_2.
$$
This is symmetric in $f_1$ and $f_2$ since the left-hand side of \eqref{Pform} is symmetric in $u_1$ and $u_2$.

We next consider $P'$. Let $f_j$ be pluriharmonic functions on $\conj D$ and 
$$
u_j=f_j\log \rho-F_j-G_j \rho^{n+1}\log\rho
$$
be the solutions to $\Delta u_j=(n+1) f_j$.
We consider the following symmetric bilinear form:
\begin{equation}\label{ipform}
\begin{aligned}
\int_{\rho>\epsilon} 
\Big(\langle \pa f_1,&\pa u_2 \rangle+\langle \pa u_1,\pa f_2 \rangle
-(n+1) f_1 f_2\Big) dv_{g}\\
&=
\int_{\rho>\epsilon} 
\Big(f_1\big(\Delta u_2-(n+1) f_2\big) + u_1\Delta  f_2 
 \Big)dv_{g}\\
&\qquad\quad\quad\quad
 +\epsilon^{-n-1}\!\!\!
 \int_{\rho=\epsilon}\Big( f_1 \cdot\nu u_2+ u_1\cdot \nu f_2\Big)\,d\sigma\\
&=\epsilon^{-n-1}\!\!\!
 \int_{\rho=\epsilon}\Big( f_1 \cdot\nu u_2+ u_1\cdot \nu f_2\Big)\,d\sigma.
\end{aligned}
\end{equation}
Since
$$
f_1\cdot\nu u_2=f_1\big( \nu f_2-(n+1)G_2 \rho^{n+1}\big)\log \rho+\text{(smooth in $\rho$)} 
$$
and
$$
u_1 \cdot\nu f_2=f_1 \cdot\nu f_2\cdot\log\rho+O(\rho^{n+2}\log\rho)+\text{(smooth in $\rho$)},
$$ 
we see that the coefficient of $\log \epsilon$ of the last line of \eqref{ipform} is
$$
-(n+1)\int_M f_1 G_2 d\sigma+
2\epsilon^{-n-1}\!\!\!\int_{\rho=\epsilon} f_1\cdot\nu f_2\,d\sigma+O(\epsilon).
$$
The first term is a constant multiple of 
$$
\int_M f_1P' f_2
$$
and thus it suffices to  show that the second term is symmetric in $f_1$ and $f_2$.
But this is clear from Green's formula:
$$
\epsilon^{-n-1}\!\!\int_{\rho=\epsilon}(f_1\cdot\nu f_2-
 f_2\cdot\nu f_1)d\sigma
=\int_{\rho>\epsilon}(f_1\Delta f_2-f_2\Delta f_1) dv_{g}=0.
\qedhere
$$
\end{proof}

\begin{rem}\rm
In the proof above, the K\"ahler condition is used to ensure that pluriharmonic functions are harmonic; thus we cannot replace $g$ by a hermitian metric in the last formula. 
We  have also used K\"ahlerness in  Green's formula, but  this  not essential since
we can replace $\pa$ by $d$ and obtain the similar formula for Riemannian Laplacian.
\end{rem}
\section{$Q$-prime curvature}

\subsection{Definition of the $Q$-prime curvature}
Recall that the $Q$-curvature is given by $Q=-\wt\Delta^{n+1}\log h_\th|_\calN$
and hence $Q=0$ if $\log h_\th$ is pluriharmonic. Such a contact from can be characterized by the pseudo-Einstein condition introduced by Lee \cite{L2}.

\begin{prop}\label{psudoEinsteinprop}
Let $M^{2n+1}$ be an embeddable CR manifold and $\rho$ be a defining function
normalized by $\th$.
 Then the following conditions are equivalent:
\begin{itemize}
\item[(1)] 
$\log(\rs/\rho)|_\calN$ is CR pluriharmonic.

\item[(2)]  The curvature and torsion of $\th$ satisfy
$$
\begin{cases}\tf \Ric_{\alpha\conj\beta}=0 &\text{if } n>1\\
\Scal_{,1}-iA_{11,}{}^{1}=0& \text{if } n=1.
\end{cases}
$$

\item[(3)]  Locally, $\th$ is volume-normalized with respect to a 
closed section of $\calK_M$.

\end{itemize}
\end{prop}
\begin{defn}\label{pe-def}
A contact form $\th$ is {\em pseudo-Einstein} if one of the equivalent conditions above holds. We will denote by $\calPE$ the space of all pseudo-Einstein contact forms on $M$.
\end{defn}

The equivalence of $(2)$ and $(3)$ has been proved in \cite[Th.\ 4.2]{L2} and \cite[Lem.\ 7.2]{H1} respectively in $n>1$ and $n=1$ cases.  
In our previous paper \cite{FH}, we call such $\th$ {\em invariant contact form}
since the Einstein condition is trivial in the case $n=1$.  However, this terminology is  featureless and we here propose to use {\em pseudo-Einstein} in all dimensions.
We believe this will not cause any confusion.

Pending the proof of this proposition, given at the end of this section, we shall give the definition of  $Q$-prime curvature.

\begin{lem}\label{defqplem} Let $M^{2n+1}$ be an embeddable CR manifold
which admits a pseudo-Einstein contact form $\th$.
Then there exists a unique defining function $\rho$ normalized by 
$\th$ such that $h_\th=\rs/\rho$ satisfies 
$$
\pa\conj\pa\log h_\th=0
$$
on the pseudoconvex side of $\calN$. For such an $h_\th$, 
$(\wt\Delta^{n+1}\log^2 h_\th)|_\calN$ is an element of $\calE(-n-1)$ and is determined by $\th$.
\end{lem}

\begin{defn}\label{def-Qp}
The {\em $Q$-prime curvature} of $\th\in\calPE$ is defined by
\begin{equation}\label{defQp}
Q'=(\wt\Delta^{n+1}\log^2 h_\th)|_\calN\in\calE(-n-1).
\end{equation}
\end{defn}

\begin{proof}
Take a fiber coordinate $z^0$ and write $\rs=|z^0|^2r(z)$ for $r\in\wt\calE(0)$. 
Then, choosing $\up$ so that $\rho=e^\up r$, we have 
$$
\log h_\th=\log |z^0|^2-\up.
$$
Since $\log h_\th$ and $\log |z^0|^2$ are CR pluriharmonic on $\calN$, so is
$\up$. Thus we may modify $\up$ to a pluriharmonic function by keeping the value on $M$.  With this $\up$, the desired function is uniquely given by $\rho=e^\up r$.

We next compute $ \wt\Delta^{n+1}\log^2 h_\th$ by using $z^0$:
$$
 \wt\Delta^{n+1}\log^2 h_\th= \wt\Delta^{n+1}\log^2 |z^0|^2
 -2\wt\Delta^{n+1}\up\log |z^0|^2+\wt\Delta^{n+1}\up^2.
$$
We know that last two terms are in $\wt\calE(-n-1)$; in fact they define $2P'\up$ and $P\up^2$.
For the first term, we have
$$
\wt\Delta\log^2 |z^0|^2=2\wt\Delta|\log z^0|^2=-
|z^0|^{-2}\langle dz^0, dz^0\rangle_{\wt g}\in\wt\calE(-1).
$$
Thus 
$
\wt\Delta^{n+1}\log^2 |z^0|^2\in\wt\calE(-n-1),
$
which gives $(\wt\Delta^{n+1}\log^2 h_\th)|_\calN\in\calE(-n-1)$.
Finally, 
we need to check that $Q'$ is independent of the ambiguity of the ambient metric. But it is exactly same as the argument in the construction of $P$.
\end{proof}

If $M$ is the boundary of a domain in $\bC^{n+1}$, by choosing a defining function $r$ satisfying $\calJ_z(r)=1+O(r^{n+2})$, we may set $\rs=|z^0|^2r$. Then  
$$
Q'=(\wt\Delta^{n+1}\log^2 |z^0|^2)|_\calN
$$
as we stated in the introduction.

The following properties of the $Q$-prime curvature are natural consequences of the definition.

\begin{prop}\label{transQP}
{\rm (i)}
If $\th\in\calPE$, then $\wh\th=e^\up\th\in\calPE$ for $\up\in\calP$
and the $Q$-prime curvature for $\wh\th$ satisfies
\begin{equation}\label{QPtrans}
\wh Q'=Q'+2P'\up+P(\up^2).
\end{equation}

{\rm (ii)} If $M$ is the boundary of a relatively compact strictly pseudoconvex domain in a Stein manifold, then
the {\em total $Q'$-curvature } 
$$
\overline Q'=\int_M Q'
$$
is independent of the choice of $\th\in\calPE$.

\end{prop}
\begin{proof}
(i)
The first assertion, $\up\in\calP$, has been shown in \cite{L1}; another proof will be given in Lemma \ref{S-lem} (iii) below.
We take a pluriharmonic extension of $\up$ and denote it by the same letter. Then we have $h_{\wh\th}=e^{-\up}h_\th$, or
$
\log h_{\wh\th}=\log h_\th-\up
$.
Thus
$$
\log^2 h_{\wh\th}=\log^2 h_\th-2 \up\log h_\th+\up^2.
$$
Applying $\wt\Delta^{n+1}$ gives \eqref{QPtrans}.

(ii)
By Theorem \ref{selfP},  $P$ and $P'$ are formally self-adjoint respectively on $\calE$ and $\calP$.  Thus noting the fact that $P1=P'1=0$, we see from
\eqref{QPtrans} that
$$
\int_M\wh Q'=\int_M Q'+2P'\up+P(\up^2)=\int_M Q'
$$
as claimed.
\end{proof}

\subsection{Total $Q$-prime curvature}
We shall write the total $Q'$-curvature as the log term coefficients in the asymptotic expansions of the integrals of natural volume forms on $D$.
For a function $f(\epsilon)$ of $\epsilon>0$ of the form
$f(\epsilon)=\varphi(\epsilon)\epsilon^{-N}+\psi(\epsilon)\log\epsilon$,
with $\varphi,\psi\in C^\infty(\bR)$, we set $
\lp f=\psi(0)$.

\begin{thm} \label{totalQpthm}
Let $r$ be a defining function of a strictly pseudoconvex domain $D$ in a Stein manifold such that the K\"ahler form $\omega=-i\pa\conj\pa\log r$ is asymptotically-Einstein.
Then the total $Q$-prime curvature of $M=\pa D$ can be written as the log term coefficients of the following two integrals:
\begin{equation}\label{Qplog}
c_n\,\conj Q'=\lp\int_{\rho>\epsilon}\|d\log r\|^2dv_g
=
\frac{-2}{(n+1)!}\lp\int_{\rho>\epsilon}\!
\left(\frac{\pa\conj\pa r}{ir}\right)^{n+1},
\end{equation}
where $c_n=(-1)^{n}/(n!)^{3}$, $\|\cdot\|$ is the norm for $g$ and $\rho$ is a defining function of $M$.
\end{thm}

In \eqref{Qplog}, the choice of $\rho$ is arbitrary. In fact, one can easily see that the leading log term in $\epsilon$ is independent of $\rho$ by writing $D$ as a product $M\times (0,\epsilon_0)$ near the boundary; see \cite[Prop. 4.1]{S}.

\begin{proof}
By the standard argument of GJMS construction, as in Lemma \ref{CRGJMSlem}, we can characterize the $Q$-prime curvature as the log term coefficient of
\begin{equation}\label{qpeq}
\wt\Delta\big(\log^2 h_\th+F+G r_\sharp^{n+1}\log r \big)=0, 
\end{equation}
where $F\in\wt\calE(0)$ with $F=O(\rho)$ and $G\in\wt\calE(-n-1)$.
$Q'$ is given by the leading term of the logarithmic singularity:
$$
G|_\calN=\frac{(-1)^{n+1}}{(n+1)!n!}Q'\in\calE(-n-1).
$$
We first compute $\wt\Delta\log^2 h_\th$.
Since $\log h_\th=\log r_\sharp-\log r$, we have
$$
\log^2 h_\th=\log^2\rs-2\log r\cdot\log\rs+\log^2 r.
$$
Applying $\rs\wt\Delta$ to each term on the right-hand side, we have
$$
\begin{aligned}
\rs\wt\Delta\log^2\rs&=\rs\wt\nabla^I\left(2{Z_I\rs^{-1}\log\rs}\right)\\
&=
2{(n+2)}\log\rs
+{2Z_IZ^I}\cdot{\rs^{-1}}\log\rs
{-2Z_IZ^I}\cdot{\rs^{-1}}\\
&=2(n+1)\log\rs+2
\end{aligned}
$$
and
$$
\begin{aligned}
\rs\wt\Delta(\log\rs\cdot\log r)&
=\rs\wt\nabla^I\left({Z_I}\rs^{-1}\log r
-\log r_\sharp\wt\nabla_I\log r\right)\\
&={(n+1)}\log r
+(Z+\conj Z)\log r+\log\rs \cdot\Delta \log r\\
&={(n+1)}(\log r+\log\rs).
\end{aligned}
$$
Here we have used $Z\log r=0$ and $\rs\wt\Delta\log r=\Delta\log r= n+1$.  The last term gives
$$
\begin{aligned}
\rs\wt\Delta\log^2r
&=\Delta\log^2 r=-\nabla^i(2\log r\cdot \nabla_i\log r)\\
&=-\|d\log r\|^2+2(n+1)\log r.
\end{aligned}
$$
In the sum of these three terms, 
the coefficients of $\log\rs$ and $\log r$ both cancel out and we get 
$$
\rs\wt\Delta\log^2h_\th=2-\|d\log r\|^2_g.
$$
Thus
we may write \eqref{qpeq} as
\begin{equation}\label{poincareQeq}
\Delta(F+G'r^{n+1}\log r)=\|d\log r\|^2-2,
\end{equation}
where $G'$ is chosen as in Lemma \ref{divlem} (ii). 
We integrate each side of this formula on the subdomain $\rho>\epsilon$.
The left-hand side gives
$$
\begin{aligned}
\lp\int_{\rho>\epsilon}\Delta(F+G'r^{n+1}\log r)dv_g
&=
\lp\int_{\rho=\epsilon}\rho^{-n-1}\nu(F+G'r^{n+1}\log r)d\sigma
\\
&=\frac{(-1)^{n}}{(n!)^3}\int_{M}Q'_\th{\th\wedge d\th^n}.
\end{aligned}
$$
While the right-hand side is
$$
\operatorname{lp}\int_{\rho>\epsilon}(\|d\log\rho\|^2-2)dv_g=
\operatorname{lp}\int_{\rho>\epsilon}\|d\log\rho\|^2 dv_g.
$$
Thus we obtain the first equality of \eqref{Qplog}.
Here we have used 
\begin{equation*}\label{lpomega}
\lp\int_{\rho>\epsilon}\omega^{n+1}=\lp\int_{\rho=\epsilon}
-i\conj\pa\log r\wedge\omega^n=0.
\end{equation*}
We can also use this formula to derive the second equality of \eqref{Qplog} since the volume form $dv_g=
\omega^{n+1}/(n+1)!$ has decomposition:
$$
\begin{aligned}
\omega^{n+1}&=(n+1)i\frac{\pa r\wedge\conj\pa r}{r^2}\wedge\omega^{n}
+\left(\frac{\pa\conj\pa r}{i r}\right)^{n+1}
\\
&=
\|\pa\log r\|^2\omega^{n+1}+\left(\frac{\pa\conj\pa r}{ir}\right)^{n+1}.\qedhere
\end{aligned}
$$
\end{proof}

\subsection{Proof of Proposition \ref{psudoEinsteinprop}}
We first introduce a $\wedge^2T^*M$-valued pseudo-hermitian invariant that unifies the tensors given in the condition (2).  Using local coordinates $(z^0,z)$, we write $\rs=|z^0|^2r(z)$ and set $\th_0=\frac{i}{2}(\pa-\conj\pa)r|_{TM}$,
which satisfies (3) by Lemma \ref{divlem} (i).

 \begin{lem} \label{S-lem}
{\rm (i)} 
For a contact form $\th$, there exists a defining function $\rho$ normalized by 
$\th$ such that
\begin{equation}\label{h-th-norm}
\wt\Delta\log h_\th=O(\rho),\quad \text{where}\quad
h_\th=\rs/\rho.
\end{equation}
Such an $h_\th$ is unique modulo $O(\rho^2)$ and hence
$$
\mathbf{S}(\th):=-\iota^*(i\pa\conj\pa\log h_\th)
$$
is formally determined by $\th$. 

\medskip
\noindent
{\rm (ii)}  If $\wh\th=e^\up\th$, $\up\in\calE$, then
$
\bfS(\wh\th\,)=\bfS(\th)+\bfP \up.
$

\medskip
\noindent
{\rm (iii)} 
In terms of the Tanaka-Webster connection, one has
$$
-(n+2)\mathbf{S}(\th)=S_{\alpha\conj\beta}\th^\alpha\wedge\th^{\conj\beta}
+S_{\alpha}\th^\alpha\wedge\th+
S_{\conj\beta}\th^{\conj\beta}\wedge\th,
$$
where
$$
\begin{aligned}
S_{\alpha\conj\beta}&=i\tf \Ric_{\alpha\conj\beta},
\\
S_{\alpha}&=({1}/{n})\Scal_{,\alpha}-i\,A_{\alpha\beta,}{}^{\beta}.
\end{aligned}
$$
Moreover, the following relation holds:
\begin{equation}\label{Sdiv}
 S_{\alpha\conj\beta,}{}^{\conj\beta}=i(n-1)S_\alpha.
 \end{equation}
\end{lem}

\begin{proof}
(i) If $\rho=e^\up r$, we get
$\log h_\th=\log |z^0|^2-\up$ and thus
$\wt\Delta \log h_\th=O(\rho)$
is reduced to the equation $\wt\Delta\up=-\wt\Delta \log h_\th+O(\rho)$ for $\up$.
Then we can apply the argument for defining $\bfP$ in Lemma \ref{Pformop} to this case.

(ii) 
If $\wh\th=e^\up\th$, then extending $\up$ off $\calN$ so that $\wt\Delta\up=O(\rho)$, we have $\log h_{\wh\th}=\log h_{\th}-\up$.
Applying $-\pa\conj\pa$ gives the transformation rule of $\bfS$.

(iii) Let $\th=e^\up\th_0$.  Since $\pa\conj\pa\log h_{\th_0}=\pa\conj\pa\log |z^0|^2=0$, we have $\bfS(\th_0)=0$. Thus the transformation law gives $\bfS(\th)=\bfP\up.$
If $n>1$, comparing $P_{\alpha\conj\beta}\up$ with the transformation law of $\tf\Ric$, we get $S_{\alpha\conj\beta}$.  Then $S_\alpha$ can be determined by $S_{\alpha\conj\beta,}{}^{\conj\beta}=i(n-1)S_\alpha$, which follows from $d\bfS(\th)=0$.
 If $n=1$, we have $S_{1\conj 1}=0$ and $S_1$ is determined by comparing $P_1\up$ with the transformation law of $\Scal_{,1}-iA_{11,}{}^{1}$.
\end{proof}

If $\log h_\th|_\calN$ is CR pluriharmonic, then we may choose $\rho$ so that $\log h_\th$ is pluriharmonic on the pseudoconvex side of $\calN$. Then we have $\pa\conj\pa\log h_\th=0$ and thus $\bfS(\th)=0$, which is equivalent to (2).
Conversely, we assume (2), or $\bfS(\th)=0$. Then writing $\th=e^\up\th_0$ as above, we have
$
\bfP\up=\bfS(\th)-\bfS(\th_0)=0,
$
which yields $\up\in\calP$.  Therefore $\log h_\th=\log |z^0|^2-\up$ is CR pluriharmonic on $\calN$.

\section{Three dimensional case}\label{three-dim}

We compute $P'$ and $Q'$ on 3-dimensional CR manifolds in terms of the Tanaka-Webster connection and use the formulas to derive the equivalence of $\conj Q'$ and the Burns-Epstein invariant.

\subsection{Pseudo-hermitian invariants}
We can greatly simplify the calculation by using invariant theory.  We give here some terminology needed for the formulation. 

 A {\em pseudo-hermitian invariant $I(\th) $ of weight $w$} is a polynomial in the components of Tanaka-Webster curvature, torsion and their iterated covariant derivatives that is
\begin{enumerate}
\item independent of the choice of admissible coframe; 
\item satisfies
$I(e^c\th)=e^{wc}I(\th)$ for any $c\in\bR$. 
\end{enumerate}
In this case, we write $I(\th) \in\calE(w)$. This formulation is used in \cite{BGS}
for the description of the heat kernel on CR manifolds.
If a pseudo-hermitian invariant $I(\th) \in\calE(w)$ satisfies
$
I(e^\up\th)=e^{w\up}I(\th)$ for any  $\up\in\calE,$
then $I(\th)$ is called a {\em CR invariant} of weight $w$.  This is a CR analoge of the conformal invariants and has been studied in \cite{F2}, \cite{Gr1}, \cite{BEG}, \cite{H2} to describe the Bergman kernel.
Note that the $Q$-curvature is a pseudo-hermitian invariant in $\calE(-n-1)$
but is not a CR invariant.

Since the choice of admissible coframe at a point has ambiguity of $U(n)$-action, we can see $I(\th)$ as a $U(n)$-invariant polynomial of the components of tensors.
Here, in view of the relation $[Z_\alpha,Z_{\conj\beta}]=-i\bh_{\alpha\conj\beta}T$, we may assume that the index $0$ does not appear in the polynomial. 
It follows from Weyl's invariant theory that $I(\th)$ is a linear combination of complete contractions of the form
\begin{equation}\label{contraction}
\contr(R^{p_1,q_1}\otimes\cdots \otimes R^{p_s,q_s}\otimes
T^{p'_1,q'_1}\otimes\cdots \otimes T^{p'_t,q'_t}),
\end{equation}
where $R^{p,q}=\nabla^{p-2,q-2}R\in\calE_{\alpha_1\dots\alpha_p\conj\beta_1\dots\conj\beta_q}(1)$ and $T^{p,q}\in\calE_{\alpha_1\dots\alpha_p\conj\beta_1\dots\conj\beta_q}$ is either $A^{p,q}=\nabla^{p-2,q}A$ or
$\conj{A}{}^{p,q}=\conj {A^{q,p}}$. The contraction is taken with respect to $\bh^{\alpha\conj\beta}$ for some pairing of holomorphic and anti-holomorphic indices.
Since a contraction changes the weight by $(-1,-1)$, the weight condition forces
$$
\sum_{j=1}^s (p_j-1)+\sum_{j=1}^t p'_j=
\sum_{j=1}^s (q_j-1)+\sum_{j=1}^t q'_j=-w.
$$

\subsection{Explicit formula of $Q'$}

We will prove \eqref{Q4form}  by using the properties given in the previous section.
Our first task is to list up all possible pseudo-hermitian invariants in volume densities.

\begin{lem}\label{lemQp}
Let $I(\th)\in\calE(-2)$ be a psudo-hermitian invariant on $M^3$.  Then there exist universal constants $c_1,c_2,c_3,c_4,c_5$ such that
\begin{equation}\label{fiveterms}
I(\th)=c_1\Delta_b \Scal+c_2 \Scal^2+c_3|A|^2+c_4 \Scal_{0}+c_5 Q,
\end{equation}
where $6Q=\Delta_b\Scal-2\Im A_{11,}{}^{11}$ is the $Q$-curvature.
\end{lem}

\begin{proof}
We first consider an expression of $I(\th)$ that does not contain the index $0$.
 By the weight condition, we see that $I(\th)$ is a linear combination of the following
  6 terms:
$$
\Scal_{1}{}^1,\ \Scal_{}{}^1{}_1,\ A_{11,}{}^{11}, \ A_{\conj{1}\conj{1},}{}^{\conj{1}\conj{1}}, \  \Scal^2, \ |A|^2.
$$
On the other hand, \eqref{commf}, \eqref{BianchiA} and the definition of $Q$ give
$$
\begin{aligned}
2A_{11,}{}^{11}&=i\Delta_b\Scal+\Scal_0-6i\,Q\\
2\Scal_{1}{}^1&=-\Delta_b \Scal+i\Scal_{0}.
\end{aligned}
$$
Thus we may replace the first 4 terms in the list by the 3 terms on the right-hand sides.
It leaves the 5 terms in \eqref{fiveterms}.
\end{proof}

Since we consider pseudo-Einstein contact forms, we may omit $Q$ and write
$$
Q'=c_1\Delta_b \Scal+c_2 \Scal^2+c_3|A|^2+c_4 \Scal_{0}.
$$
Our remaining task is to identify the universal constants.
We can do this just by considering $Q'$ on a flat CR manifold, for which the ambient metric is explicitly given.
This approach has been used in our computation of the $Q$-curvature in \cite{H1}.

Let $M_0$ be the real hypersurface in $\bC^2$ given by the defining function
$$
\rho_0=w+\conj w-|z|^2.
$$
We write $w=s+it$ and use $(z,t)\in\bC\times\bR$ as coordinates of the surface.
Then $\rho_0$ gives the contact form
$$
\th_0=(i/2)(\pa-\conj\pa)\rho=-dt-(i/2)\conj zdz+(i/2)zd\conj z
$$
and we may choose $dz$ as an admissible coframe.  The dual frame of $\th_0$, $dz$, $d\conj z$ is
$$
-\pa_t,\quad
Z_1=\pa_{z}+\conj{z}\pa_{w}=\pa_{z}-(i/2)\conj z\pa_t,\quad Z_{\conj{1}}=\conj{Z_1}.
$$
This frame is parallel with respect to the connection for $\th_0$.
Let
$$
z=\zeta^1/\zeta^0,\quad w=\zeta^2/\zeta^0
$$
and set $r_\sharp=\zeta^0\conj\zeta{}^2+\zeta^2\conj\zeta{}^0-
\zeta^1\conj\zeta{}^1$. Then the ambient K\"ahler Laplacian is 
$$
\wt\Delta=\pa_{0}\pa_{\conj 2}+\pa_{2}\pa_{\conj 0}-
\pa_{1}\pa_{\conj 1},\quad \text{where }\pa_I=\pa/\pa\zeta^I.
$$
We will use the frame $\pa_I$ and the coframe $d\zeta^I$ to define the components of tensors on the ambient space
so that the covariant derivatives $\wt\nabla_I$ agree with $\pa_I$.

We consider a pseudo-Einstein contact form given by the scaling $\th=e^\up\th_0$, where
$$
\up=f+\conj{f}, \quad f=a(z^2+z)+b(1+i)w\ \text{ for }a,b\in\bR.
$$ 
We will use $O(1)$ to denote terms that vanish at $(z,w)=(0,0)\in M_0$ or at
$e_0=(\zeta^I)=(1,0,0)\in\calN$. 

\begin{lem}
The $Q$-prime curvature of $\th=e^\up\th_0$ satisfies
\begin{equation}\label{Qpab}
Q'_\th=8(a^2+b^2)+O(1).
\end{equation}
\end{lem}
\begin{proof}
Since 
$h_\th=|z^0|^2e^{-\up}$,  we have, on $\calN$,
$$
\begin{aligned}
Q'&=\wt\Delta^2\log^2 h_\th\\
&=\wt\Delta^2(\log z^0-f+\log\conj z^0-\conj f)^2\\
&=2\wt\Delta^2(|\log z^0|^2-2\Re( f\log \conj z^0)+|f|^2).
\end{aligned}
$$
Clearly, $\wt\Delta|\log z^0|^2=0$ and  
$$
\wt\Delta^2( f\log \conj z^0)=
\pa_{IJ}f\cdot\pa^{IJ}\log \zeta^{\conj 0}=-(\zeta^{\conj 0})^{-2}\pa_{22} f=0.
$$
Thus we have 
$$
Q'=2\wt\Delta^2|f|^2=2\pa_{IJ}f\cdot\pa^{IJ}\conj f.
$$
The components $F_{IJ}=\pa_{IJ}f(e_0)$ vanish except for 
$$
F_{10}=F_{01}=-a,\quad F_{11}=2a,\quad F_{20}=F_{02}=-b(1+i)
$$
and their conjugates. So, at $e_0$, we have
$$
\wt\Delta^2|f|^2=F_{11}\conj{F_{11}}+2F_{20}\conj{F_{02}}
=4(a^2+b^2).
$$
This is the value at $(0,0)$ of the density $Q'$ identified with a function with respect to $\th_0$. However,
since $\up=O(1)$, the value does not change by the replacement of $\th_0$ by $\th=e^\up\th_0$.
\end{proof}

We next apply Proposition \ref{transf-prop} to evaluate 
$\Delta_b\Scal$, $\Scal^2$, $|A|^2$
and $\Scal_{0}$ at $(0,0)$.
We first compute the derivatives of $\up$:
$$
\begin{aligned}
\up_0&=2b\qquad
&&\up_1=a(2z+1)+b(1+i)\conj z=a+O(1)
\\
\up_{11}&=2a\qquad
&&\up_{1\conj 1}=\conj{\up_{\conj 1 1}}=
b(1+i).
\end{aligned}
$$
From the last formula, we get $\Delta_0\up=-2b$, where $\Delta_0$ be the sub-Laplacian for $\th_0$.
Proposition  \ref{transf-prop} then gives
$$
\begin{aligned}
\Scal&=2\Delta_0\up-2\up_1\up^1=-4b-2\big|a(2z+1)+b(1+i)\conj z\big|^2,
\\
A_{11}&=i\up_{11}-i(\up_1)^2=i(2a-a^2)+O(1),
\\
\Delta_b\Scal&=\Delta_0\Scal-(\Delta_0\up)\Scal+\up_1\Scal^1+\up^1\Scal_1\\
&=-8a^2(a-2)-8a^2b+O(1),
\\
\Scal_0&=i\up_1\Scal^1-i\up^1\Scal_1+\frac{-2}{3}(\up_0+\Im \up_1{}^1)\Scal=8b(a^2+b).
\end{aligned}
$$
Therefore, at $(z,w)=(0,0)$, 
\begin{equation}\label{trans-terms}
\begin{aligned}
\Delta_b \Scal&=-8a^2(a+b-2)
\qquad
&&\Scal^2=4(a^2+2b)^2
\\
\Scal_{0}&=8b(a^2+b)
\qquad
&&|A|^2=a^2(a-2)^2.
\end{aligned}
\end{equation}
Observe that 
the four polynomials on the right-hand sides are linearly independent and
$$
\Delta_b \Scal+\frac12\Scal^2-2|A|^2=8(a^2+b^2).
$$
Since the right-hand side agrees with $Q'$ given in \eqref{Qpab},
we obtain \eqref{Q4form}.

\subsection{Explicit formula of $P'$}\label{ss-explicit-pp}
Having found $Q'$, we can compute $P'$ by using the transformation rules
\eqref{trans-PP} and
\eqref{QPtrans}.

\begin{prop}
Let $\th$ be a pseudo-Einstein contact from on $M^3$. Then $P'$ can be extended to a self-adjoint  differential operator $D$ on $\calE$:
$$
Df=\Delta_b^2f-\Re(\Scal f_1-2iA_{11}f^1)^{,1}.
$$
\end{prop}
\begin{proof}
For $\up\in\calP$, consider a family of contact forms $\th_\epsilon=e^{\epsilon\up}\th$.  Then, denoting by $\delta=(d/d\epsilon)|_{\epsilon=0}$, we have
from Proposition \ref{transf-prop} that
$$
\begin{aligned}
\delta(\Delta_b\Scal)&=2\Delta_b^2\up-\Scal\Delta_b\up+2\Re(\Scal_{,1}\up^1),
\\
\delta(\Scal^2)&=4\Scal\Delta_b\up,
\\
\delta(|A|^2)&=-2\Re(iA_{11}\up^{11}),
\end{aligned}
$$
and hence
$$
\delta Q'
=2\Delta_b^2\up+\Scal\Delta_b\up+2\Re(\Scal_{1}\up^1+2iA_{11}\up^{11}).
$$
On the other hand, using $\Scal_{1}=iA_{11,}{}^{1}$, we have
$$
\begin{aligned}
\Re\big( \Scal \up_1 -2i A_{11}\up^1
\big)^{,1}
 &=
\Re\big( \Scal^1 \up_1+\Scal \up_1{}^1\\
&\qquad\qquad
-2i A_{11,}{}^1\up^1
-2i A_{11}\up^{11}\big)\\
& =-\frac12 \Scal\Delta_b\up-
\Re\big(\Scal_1 \up^1+2i A_{11}\up^{11}\big).
\end{aligned}
$$
Thus we get $\delta Q'=2D\up$, which implies $P'=D$. 
The self-adjointness of $D$ follows from the divergence formula
$
\int_M \tau_{1,}{}^1=0
$
for $\tau_1 \in\calE_1(-1)$ and the Leibniz rule.
\end{proof}

We next derive the formula of $P'$ for general contact form by using the transformation rule
\eqref{trans-PP}.  We modify $D$ by adding two terms that vanish for $\th\in\calPE$:
$$
P_{\alpha,b}f=Df+\Re (\alpha S^1f_1)+b\,Qf,
$$
where $\alpha\in\bC$ and $b\in\bR$. 

\begin{lem}\label{transflemma}
For any scaling $\wh\th=e^\up\th$ and $f\in \calE$, one has
$$
\begin{aligned}
\wh P_{\alpha,b}f=&P_{\alpha,b}f+P(\up f)-\up Pf+(b-1)f P\up
\\
&+\Re \big[(2-3\alpha)f^1P_1\up-4\up^1 P_1f\big].
\end{aligned}
$$
\end{lem}

\begin{proof}
A straightforward computation using Proposition \ref{transf-prop} gives
$$
\wh Df=Df+2\Re(2\up^1{}_1f_1+\up_1{}^1f_1+\up^1f_{11}-\up_1f_1{}^1)^{,1}.
$$
Expanding the second term gives 3rd order derivatives of $f$ and $\up$. Such derivatives can be 
replaced by $P_1$, $P^1$ and lower order terms by using
$$
\begin{aligned}
\up^1{}_1{}^1&=P^1\up-i\up^1{}_0\\
\up_{11}{}^1&=P_1\up+2i\up_{10}+\Scal\up_1\\
\up_1{}^{11}&=P^1\up+iA^{11}\up_1,
\end{aligned}
$$
which follow from $P_1\up=\up^1{}_{11}+iA_{11}\up^1$ and \eqref{commf}.
Thus we get
$$
\begin{aligned}
\wh D f&=Df+2\Re(2\up^1{}_1f_1+\up_1{}^1f_1+\up^1f_{11}-\up_1f_1{}^1)^{,1}
\\
&=Df+2\Re(
3 f_1P^1\up-2i\up^1{}_0f_1+2i\up^1f_{10}\\
&\qquad\qquad\qquad\ \ +2\up^1{}_1f_1{}^1+\up^{11}f_{11}+\Scal \up^1 f_1
).
\end{aligned}
$$
Applying the same method to $Pf=(P^1f)_{,1}$ gives
$$
\begin{aligned}
P(\up f)
&=
\up Pf+f P\up+
2\Re\big(
2\up^1P_1 f+2f^1P_1\up\\
&\qquad 
+2i\up_{10}f^1+2i\up^1f_{10}
+2\up_1{}^1 f^1{}_1+\up^{11}f_{11}+\Scal \up^1f_1
\big).
\end{aligned}
$$
Thus
$$
\wh D f=Df+P(\up f)-\up Pf-f P\up+2\Re\big(f_1P^1\up-2\up^1P_1f\big).
$$
Combining with the transformation rule of $S_1$ and $Q$, we get the lemma.
\end{proof}

By this lemma, we see that 
$
\wh P_{\alpha,b}f=P_{\alpha,b}f+ P(\up f)$ holds for $f\in\calP$
and $\up\in\calE$ if and only if $\alpha=2/3$ and $b=1$.
Therefore, we get
$$
P'f=Df+2/3\,\Re (S^1f_1)+Qf.
$$
We know that $P'$ is formally self-adjoint on $\calP$ if $M$ is the boundary of a strictly pseudoconvex domain in a Stein manifold.  It follows that the first order operator $\Re (S^1f_1)$ is also formally self-adjoint on $\calP$.  We can prove this fact directly by using $i\pa\conj\pa\log h_\th$ as follows:
Recall that
$
-3\iota^*(\pa\conj\pa\log h_\th)=iS_1\th^1\wedge\th+iS_{\conj 1}\th^{\conj 1}\wedge\th
$
and so
$$
3(\conj\pa-\pa)\varphi\wedge \pa\conj\pa\log h_\th
=2\Re(S^1 \varphi_1)\th\wedge d\th\quad\text{on }\ M.
$$
Taking pluriharmonic extensions of $\varphi$ and $\psi$, we get
$$
\begin{aligned}
2\int_M\psi\Re(S^1 \varphi_1)\th\wedge d\th
&=
3\int_Dd\big(\psi(\conj\pa-\pa)\varphi\wedge \pa\conj\pa\log h_\th\big)\\
&= 3
\int_D 
(\pa\psi\wedge\conj\pa \varphi+\pa\varphi\wedge\conj\pa\psi )\wedge \pa \conj\pa\log h_\th.
\end{aligned}
$$
The last integral is symmetric in $\varphi$ and $\psi$; thus
$\Re(S^1 \varphi_1)$ is formally self-adjoint on $\calP$. We remark that this proof does not use the Einstein K\"ahler metric on the domain.

\subsection{Burns-Epstein invariant} We start by recalling the definition of the invariant $\mu$ from Burns-Epstein \cite{BE1}.  Suppose that $M$ is compact and has trivial holomorphic tangent bundle $T^{1,0}$.
Then we may take a global admissible coframe $\th, \th^1,\th^{\conj 1}$ such that
$d\th=i\th^1\wedge\th^{\conj 1}$.
Using the globally defined connection form  $\omega_1{}^1$, we define a 3-form
$$
\wt T C=\frac{i}{8\pi^2}
\left[\Big(-\frac{2i}3 \,d\omega_1{}^1+\frac16\, d(\Scal\th)\Big)\wedge \omega_1{}^1-2|A|^2\th\wedge \th^{\conj 1}\wedge\th^{1}
\right].
$$
While the definition depends on a choice of coframe,
it is shown that the de Rham class of $\wt TC$ is determined by the CR structure and
$$
\mu=\int_M\wt T C
$$
gives a CR invariant, which is  called the {\em Burns-Epstein invariant.}
We can easily generalize the definition to the case where $c_1(T^{1,0})$ is a torsion, i.e. $c_1(T^{1,0})=0$ in $H^2(M,\bR)$. 
In this case, instead of taking a global coframe, we can choose local coframe such that the transition functions are given by constants.  Then $\omega_1{}^1$ does not depend on the choice of the frame and $\wt TC$ is well-defined.
Note that this condition holds if $M$ admits a pseudo-Einstein contact form;
see \cite[Prop.\ D]{L2}, the proof also holds for the case $n=1$.

We now assume that $M$ admits a pseudo-Einstein contact form.  Then there is a natural choice of frame so that $\wt{T}C$ agrees with a pseudo-hermitian invariant in $\calE(-2)$ for $\th\in\calPE$.
To formulate it, we give a characterization of pseudo-Einstein condition in terms of the connection form, which holds for all dimensions.

\begin{lem} A contact form $\th$ on  $M^{2n+1}$ is pseudo-Einstein if and only if  there exists an admissible local coframe $\th$, $\th^\alpha$, such that
\begin{equation}\label{normalframe}
h_{\alpha\conj\beta}=\delta_{\alpha\conj\beta}
\quad\text{and}\quad 
\omega_\alpha{}^\alpha+(i/n)\Scal\th=0.
\end{equation}
\end{lem}

\begin{proof} We follow the proof of \cite[Th.\ 4.2]{L2}.
We take an admissible coframe such that $h_{\alpha\conj\beta}=\delta_{\alpha\conj\beta}$; then $\th$ is normalized by
$$
\zeta=e^{i\psi}\th\wedge\th^1\wedge\cdots\wedge\th^{n+1}
$$
for any real function $\psi$. If $\th$ is pseudo-Einstein, we may choose $\psi$ so that $\zeta$ is closed. Then, replacing $\th^1$ by $e^{i\psi}\th^1$, we may assume that $\zeta=\th\wedge\th^1\wedge\cdots\wedge\th^{n+1}$ is closed, so that
$
0=d\zeta=-\omega_\alpha{}^\alpha\wedge\zeta
$.
Since $\omega_\alpha{}^\alpha$ is pure imaginary due to $dh_{\alpha\conj\beta}=0$, we see that
$
\omega_\alpha{}^\alpha+iu\th=0
$
for a real valued function $u$.  Applying $d$ to this formula gives
$$
0=d\omega_\alpha{}^\alpha+iud\th+idu\wedge\th=(\Ric_{\alpha\conj\beta}-u h_{\alpha\conj\beta})\th^\alpha\wedge\th^{\conj\beta}
\mod \th.
$$  Thus we get $\Scal-nu=0$ and hence \eqref{normalframe} follows. 
The converse follows from
$d\zeta=-\omega_\alpha{}^\alpha\wedge\zeta=0$
for $\zeta=\th\wedge\th^1\wedge\cdots\wedge\th^{n+1}$.
\end{proof}

Now we specialize the proposition to the case $n=1$.  Then the transition function between two admissible frame
$\th^1$ satisfying \eqref{normalframe} is given by a constant in $S^1$; hence $\omega_1{}^1$ is globally defined.
Substituting $\omega_1{}^1=-i\,\Scal\th$ and $\th^1\wedge\th^{\conj{1}}=-id\th$ into $\wt T C$ gives
$$
\wt T C
=\frac{-1}{16\pi^2}(\Scal^2-4|A|^2)\th\wedge d \th.
$$
We have thus proved 
\begin{thm}
On a compact pseudo-Einstein manifold $M^3$, one has
\begin{equation}\label{inv-int}
\int_M Q' =-8\pi^2\mu(M).
\end{equation}
\end{thm}

In conformal geometry, Alexakis \cite{A} showed that any, local, volume form-valued, Riemannian invariant $I(g)$ whose integral is a conformal invariant for any 
 compact manifold can be decomposed into three parts:
$$
I(g)=c\,\operatorname{Pfaffian}+\text{(local conformal invariant)}+\text{(exact form)}.
$$
As a natural analogue of Alexakis' theorem in CR setting, we make the following conjecture:
If $I(\th)\in\calE(-n-1)$ is a local pseudo-hermitian invariant such that the integral is independent of the choice of $\th\in\calPE$, then
$$
I(\th)=c\, Q'+\text{(local CR invariant)}+\text{(exact form)}
$$
for any $\th\in\calPE$.

In case $n=1$, we can verify this by a direct computation:
By Lemma \ref{lemQp}, we know that $I(\th)$ modulo exact forms is of the form
$c_1 \Scal^2+c_2 |A|^2$. On the other hand, we know that $2Q'=\Scal^2-4|A|^2+\text{(exact form)}$.  Since $\int_M\Scal^2$ is not a CR invariant for $\th\in\calPE$,
we  have
$$
I(\th)=c_1 Q'+\text{(exact form)}.
$$
Here we do not have the CR invariant part, simply because there are no CR invariant in
$\calE(-2)$ on $M^3$; see \cite{Gr1}, \cite{HKN}, \cite{H2}.
The conjecture is open for higher dimensions.

\subsection{Variational formula}\label{variation-ss}
We finally derive the variational formula \eqref{QPvar}.
Let $Z_1$ be a frame of $T^{1,0}$. Then we may define a one-parameter family of CR structure by the frame
$$
Z_1^\epsilon=Z_1+\varphi^\epsilon{}_{1}{}^{\conj 1}Z_{\conj 1}
$$
for $\varphi^\epsilon{}_{11}\in\calE_{11}(1)$ which smoothly depends on $\epsilon\in\bR$.
Note that this family is independent of the choice of frame $Z_1$ and is intrinsically defined from  $\varphi^\epsilon{}_{11}$.  We will denote the family of CR manifolds by $\{M_\epsilon\}$.  Then, by \cite[Prop.\ 3.3]{BE1} and \cite[Prop.\ 2.6]{CL},
we have
$$
\frac{d}{d\epsilon}\Big|_{\epsilon=0}
\mu(M_\epsilon)=-\frac{1}{4\pi^2}\Re \int_{M} \varphi_{11}\calO^{11},
$$
where $
\varphi_{11}=({d}/{d\epsilon})\big|_{\epsilon=0}\varphi^\epsilon{}_{11}$ and $\calO^{11}\in\calE^{11}(-3)$ is the Cartan curvature of $M$:
$$
\calO^{11}=\frac16 \Scal^{11}-\frac{i}2\Scal A^{11}-A^{11}{}_{,0}+\frac{2i}{3}A{}^{11}{}_{,1}{}^1.
$$ 

We  consider the special case when the deformation is given by a family of strictly pseudoconvex hypersurfaces
$\{M_\epsilon\}$ in a fixed complex manifold
$X$.  Such deformations are called {\em Kuranishi wiggles} and it is known that 
the first variation $\varphi_{11}$ can be written as the second derivative of a density:
$$
\varphi_{11}=P_{11}f:=(\nabla_{11}+iA_{11})f, \quad f\in\calE(1).
$$
Note that $P_{11}\colon\calE(1)\to\calE_{11}(1)$ is a CR invariant operator and so is the formal adjoint
$$
P^*_{11}\colon\calE^{11}(-3)\to\calE(-3).
$$
These are respectively the first and the last of the deformation complex, which is the BGG complex for the adjoint representation of $\SU(n+1,1)$; see \cite{CSS}.
Using the double divergence, we can define a CR invariant $\calO=P^*_{11}\calO^{11}\in\calE(-3)$.
It is real valued, since
the Bianchi identity for the Cartan curvature \cite[Prop.\ 3.1]{CL} gives
$
\Im P^*_{11}\calO^{11}=0
$.  
Recall from \cite{Gr1} that there is a unique CR invariant in $\calE(-3)$
and $\calO$  agrees (up to a constant multiple) with the obstruction function 
$\eta|_\calN\in\calE(-3)$ defined by \eqref{Ric-obs}.  If $\varphi_{11}=P_{11}f$, $f\in\calE(1)$, then integration by parts gives
\begin{equation}\label{var-tQ}
\frac{d}{d\epsilon}\Big|_{\epsilon=0}
\mu(M_\epsilon)=-\frac{1}{4\pi^2}\Re \int_M f\calO.
\end{equation}
Observe that the imaginary part of $f$ does not contribute to the integral since $\calO$ is real.  This is consistent with the fact that $\Im f$ corresponds to
the trivial deformation given by the pull-back by contact diffeomorphisms; see \cite[Lem.\ 3.4]{CL}.  If we further assume that each $M_\epsilon$ admits a pseudo-Einstein contact form,
then $\conj Q' =-8\pi^2\mu$ and we obtain \eqref{QPvar}.

\section{Appendix: Hartogs' theorem for pluriharmonic functions}

In Theorem \ref{selfP}, we have assumed that $M$ is the boundary of a domain $D$ in a Stein manifold.  But we can weaken it since we have only used the following two facts:
\begin{enumerate}
\item
$D$ admits asymptotically Einstein K\"ahler metric;
\item
Any
$f\in\calP$ has pluriharmonic extension to $D$.
\end{enumerate}

We know by Theorem \ref{CYthm} that (1) holds if $\calK_D>0$.
While the asymptotically Einstein condition seems to be weaker than the Einstein condition, so far, we do not have examples beyond this class.

For the Hartogs extension (2), we can give a more general result
which is suggested by Takeo Ohsawa.

\begin{thm}\label{phextthm}
 Let $D$ be a bounded strictly pseudoconvex domain with connected boundary in a K\"ahler manifold $X$ of dimension $n+1\ge3$.  Then CR pluriharmonic functions on $\pa D$ can be extended to pluriharmonic functions on $D$.
\end{thm}

\begin{proof}
We recall two theorems from the $L^2$-theory of several complex variables.  Here $D$ satisfies the assumption of the theorem above;
 more general statements can be found, respectively, in  \cite[Cor.\ 7]{Oh1} and  \cite[Satz 7]{GR}.  

\medskip
\noindent
{\bf Hartogs' theorem for holomorphic forms.} 
Let  $K\subset D$ be a compact subset such that $D\setminus K$ is connected.  If $p\le n-1$, then 
any holomorphic $p$-form on $D\setminus K$ has holomorphic extension to $D$.

\medskip
\noindent
{\bf $\conj\pa$-Poincar\'e lemma with compact support.}
If $\eta$ is a closed $(0,1)$-form with compact support, then there exists an
$h\in C_0^\infty(D)$ such that $\eta=\conj\pa h$.

\medskip
For $f\in\calP$, take an extension $\wt f$ such that $\pa\conj\pa\wf=0$ outside a compact set $K\subset D$. 
 Then
$\pa\wf$ is holomorphic on $D\setminus K$ and thus we may find a holomorphic 1 form $\varphi$ such that
$u:=\pa\wf-\varphi$ has compact support. Applying $\pa$ gives
$\pa u=-\pa\varphi$. The left-hand side has compact support, while the right-hand side is holomorphic; hence $\pa u=0$ on $D$.  Thus we can apply (the conjugate of) $\conj\pa$-Poincar\'e lemma 
to find an $h\in C^\infty_0(D)$ such that $\pa h=u$.
Then, in view of
$$
\conj\pa\pa h=\conj\pa u=\conj\pa\pa\wt f-\conj\pa\varphi=\conj\pa\pa\wt f,
$$
we see that $\wt f-\Re h$ is pluriharmonic on $D$ and its boundary value is $f$.
\end{proof}

The restriction on dimension, $n\ge2$, is caused by the condition $p\le n-1$ in Hartogs' theorem for holomorphic $p$-forms.  This should not be essential but we will not go into the detail of this point in this paper.  

\section*{Acknowledgements}
We would like to thank Jeffrey Case and Paul Yang for generously sharing ideas on $P'$ and $Q'$.
  We would also like to thank Mike Eastwood, Robin Graham, Yoshihiko Matsumoto,
  Kimio Miyajima,
Bent \O rsted, Takeo Ohsawa and
Michael Range for helpful discussions. 
We are grateful to the referees for a number of helpful suggestions for improvement in the article.

\end{document}